\documentclass[preprint,12pt,authoryear]{elsarticle}

\usepackage[T1]{fontenc}
\usepackage{amsmath}
\usepackage{amssymb}
\usepackage{lscape}
\usepackage{amsthm}
\usepackage[left=1.5cm, right=1.5cm, top=2.5cm]{geometry}
\usepackage{hyperref}
\usepackage{algorithm}
\usepackage{bbm}
\usepackage{booktabs}
\usepackage{algorithm}
\usepackage{algpseudocode}
\usepackage{tikz}
\usepackage{caption}
\usepackage{subcaption}
\usepackage{fancyhdr}
\usepackage[inline]{enumitem}
\usepackage{comment}
\usepackage{nicefrac}
\usepackage{bm}
\usepackage{mathrsfs}
\usepackage{multirow}
\usepackage{graphicx}
\usepackage[utf8]{inputenc}
\usepackage{cancel}
\usepackage{mathtools}
\usepackage{natbib}

\newtheorem{theorem}{Theorem}
\newtheorem{lemma}[theorem]{Lemma}
\newtheorem{proposition}[theorem]{Proposition}
\newtheorem{corollary}{Corollary}[theorem]

\newtheorem{definition}[theorem]{Definition}
\newtheorem{remark}{Remark}
\newtheorem{assumption}{Assumption}
\newtheorem{example}[theorem]{Example}

\definecolor{mypink}{RGB}{231,84,128}
\newif\ifshowcomments
\showcommentsfalse
\ifshowcomments
  \NewDocumentCommand{\mengqi}{m}{{\color{mypink} [\textbf{Mengqi}: #1]}}
\else
  \NewDocumentCommand{\mengqi}{m}{}
\fi

\newcommand{\osc}{\operatorname{osc}}

\begin{document}




\title{Imprecise Markov Semigroups and their Ergodicity} 


\author{Michele Caprio}\address{Department of Computer Science, The University of Manchester}

\author{Mengqi Chen}\address{Department of Statistics, University of Warwick}

\begin{abstract}
We introduce the concept of an imprecise Markov semigroup $\mathbf{Q}$. It is a tool that allows us to represent ambiguity around both the transition probabilities and the invariant measure of a continuous-time Markov process via a collection of Markov semigroups, each associated with a (possibly different) Markov process. 
We use techniques from topology, geometry, and probability to analyze ergodic limits under model uncertainty encoded by \(\mathbf Q\). We establish long-term bounds that are uniform in the initial state and identify regimes in which the imprecision in these bounds collapses asymptotically. Our results are proved in progressively more general settings. We first assume that \(\mathbf Q\) is compact and that the state space is Euclidean or a Riemannian manifold, working with a fixed bounded observable. We then allow the state space to be standard Borel, while keeping \(\mathbf Q\) compact and the observable fixed. Finally, we drop compactness and work on Polish metric spaces of finite diameter, where we treat arbitrary bounded Lipschitz observables.
The importance of our findings for the fields of artificial intelligence and computer vision is also discussed at a high level; In particular, in the study of how the probability of an output evolves over time as we perturb the input of a convolutional autoencoder.
\end{abstract}


%
%
\maketitle






\section{Introduction}\label{intro}

When dealing with a stochastic phenomenon whose future evolution is independent of its history, it is customary to model its behavior via a stochastic process enjoying the Markov property. Loosely, this means that the law of the realization $X_T$ of the process $(X_t)_{t\geq 0}$ at time $T$, given what happened up to time $s <T$, is equivalent to the law of $X_T$, given only $X_s$.
We can of course express this property in a slightly more general way, by saying that the expectation of a generic bounded measurable real-valued function $\tilde f$ of $X_T$, given what happened in the past up to time $s<T$, is equal to the expectation of $\tilde{f}(X_T)$, given only $X_{s}$.
Notice that in the present paper we perpetrate a terminology abuse, and refer to real-valued functions on the state space $E$ as ``functionals''. In addition, we only consider continuous-time processes $(X_t)_{t\geq 0}$.


Two elements play a key role in a process enjoying the Markov property. The first is the 
law of the initial state $X_0$, and the second is the transition probability from the initial state $X_0$ to any state $X_t$ of the process, i.e. the probability that, given that the process begins at $X_0$, it ``reaches a certain value'' at time $t$. 
In many applications, relaxing the hypothesis that one or both of them are precise is natural, and leads to less precise but more robust results.

Imprecise probability \citep{intro_ip,decooman,walley} is a field whose techniques naturally lend themselves to be used for this purpose. In fact, imprecise Markov processes are a flourishing area of the imprecise probability literature \citep{CROSSMAN20101085,sum-prod,de_Cooman_Hermans_Quaeghebeur_2009,delgado2011efficient,jaeger,krak,hitting,tjoens,trevizan2007planning,trevizan2008mixed,troffaes2019},\footnote{Much of the existing literature focuses on discrete-time  processes.} with numerous applications e.g. in engineering \citep[Chapter 4]{engin}, including aerospace engineering \citep[Chapter 5]{engin2}, and economics \citep{boppi,denk2,DENK}. As the name suggests, imprecise Markov processes allow to take into account imprecision in the initial and transition probabilities (and also in the invariant measure, when it exists) of a Markov process.

In this paper, we introduce a new concept that we call an {\em Imprecise Markov Semigroup}. It is a collection $\mathbf{Q}$ of (precise) Markov semigroups, each associated with a Markov process.\footnote{Here, by ``precise'' we refer to the classical notion of a Markov semigroup.} Modeling the imprecision on the transition probabilities and on the invariant measure (whenever it exists) via such a collection, instead of using lower previsions \citep{decooman} or Choquet integrals \citep{choquet}, allows us to
combine geometric and probabilistic techniques based on the Markov diffusion operators literature \citep{bakry} that make the study of the limiting behavior of imprecise Markov processes (captured by the limit behavior of the imprecise Markov semigroup) easier.
In particular, we are able to relate such a behavior to the curvature of the state space -- that is, the space $E$ that the elements of the process $(X_t)_{t\geq 0}$ take values on. This also links the present work to the study of the geometry of uncertainty \citep{cuzzolin_book}.

Our goal is to study what we call the {\em ergodic property} of an imprecise Markov semigroup. Throughout the paper we use ``ergodicity'' in the semigroup/long-term sense: We study the asymptotic behavior, as $t\to\infty$,
of conditional expectations of observables under the Markov evolution, i.e. $P_t f(x)=\mathbb E[f(X_t)\mid X_0=x]$.
In Sections~\ref{riemann} and \ref{idmt-section}, the bounds we derive are formulated for a \emph{fixed} (bounded) test functional $\tilde f$;
Accordingly, the conclusions should be read as long-term bounds for the expectation of $\tilde f$ whose limiting constants do not depend on the initial state, 
rather than as full distributional convergence -- which would correspond to convergence for all $f\in B(E)$.
In contrast, in Section~\ref{subsec:wass-sublinear} we work on the larger class $\mathrm{Lip}_b(E)$ of bounded Lipschitz functionals and obtain uniform-in-$x$ limits
for the static upper and lower envelopes $\overline T_t$ and $\underline T_t$ of $\mathbf{Q}$, respectively, and for the time-consistent (upper/lower) Nisio semigroups when switching is allowed. Such results are derived
for every bounded Lipschitz observable $f$, which is closer in spirit to distributional stability.
Finally, we emphasize that our limits are \emph{not} Ces\`aro time averages:
We do \emph{not} prove statements of the form $\frac1T\int_0^T f(X_t) \mathrm dt \to \int f \mathrm d\mu$.
Rather, we establish convergence of the \emph{semigroup} as $t\to\infty$ (equivalently, convergence of one-time marginals and/or of their upper/lower envelopes when switching is allowed). When appropriate (e.g. under standard invariance/ergodicity assumptions),
classical ergodic theorems can then be invoked to relate such semigroup convergence to time average statements, but this is not the
object of the present results.

The similarities and differences between our results and the existing ones in the imprecise Markov literature and in imprecise ergodic theory \citep{KozineUtkin2002intervalValuedFiniteMarkovChains,HermansDeCooman2012ergodicUpperTransitionOperators,Skulj2013classificationInvariantDistributions,Skulj2015efficientComputationCTIMC,cerreia,DECOOMAN201618,dipk,ergo_me}, are discussed in
Sections \ref{prelim} and \ref{riemann}, and \ref{subsec:wass-sublinear}.

Besides being interesting in their own right (and because they provide new techniques to the fields of imprecise Markov processes), we believe that our findings will prove useful in future research in machine learning, artificial intelligence, and computer vision. Let us add a high-level discussion on this topic. In recent works on convolutional autoencoders, images (or, more in general, inputs from an input space $\mathcal{X}$) are ``projected'' onto a low-dimensional manifold $E$ by an encoding function $\phi_\text{enc}:\mathcal{X} \rightarrow E$. Different ``portions'' of the surface $\mathcal{S}(E)$ of such a manifold capture different features of the input image \citep{YU202348}. For instance, a portion  of the surface may correspond to the feature ``dog'', so that dog images are projected there. This means that we can consider a partition $\mathscr{E}=\{E_j\}_{j=1}^J$ of $\mathcal{S}(E)$ whose elements correspond to ``feature portions''. For example, $J$ could be equal to the cardinality $|\mathcal{Y}|$ of the output space $\mathcal{Y}$ in the case of classification problems. Then, a decoding function $\phi_\text{dec}:E \rightarrow \Delta_\mathcal{Y}$ returns the probabilities of $y_j \in\mathcal{Y}$ being the correct output for the input $x \in\mathcal{X}$. Here $\Delta_\mathcal{Y}$ denotes the space of countably additive probabilities on $\mathcal{Y}$. For example, in a classification case, $(\phi_\text{dec} \circ \phi_\text{enc})(x)=(P(y_1),\ldots,P(y_J))^\top$, where $\mathcal{Y}=\{y_1,\ldots,y_J\}$ is the label space, and $P(y_j)$ is the probability estimated by the model that $y_j$ is the correct label for input $x$, $j\in\{1,\ldots,J\}$. A visual representation is given in Figure \ref{IMSG_heuristic}.

\begin{figure}[h]%
\centering
\includegraphics[width=.7\textwidth]{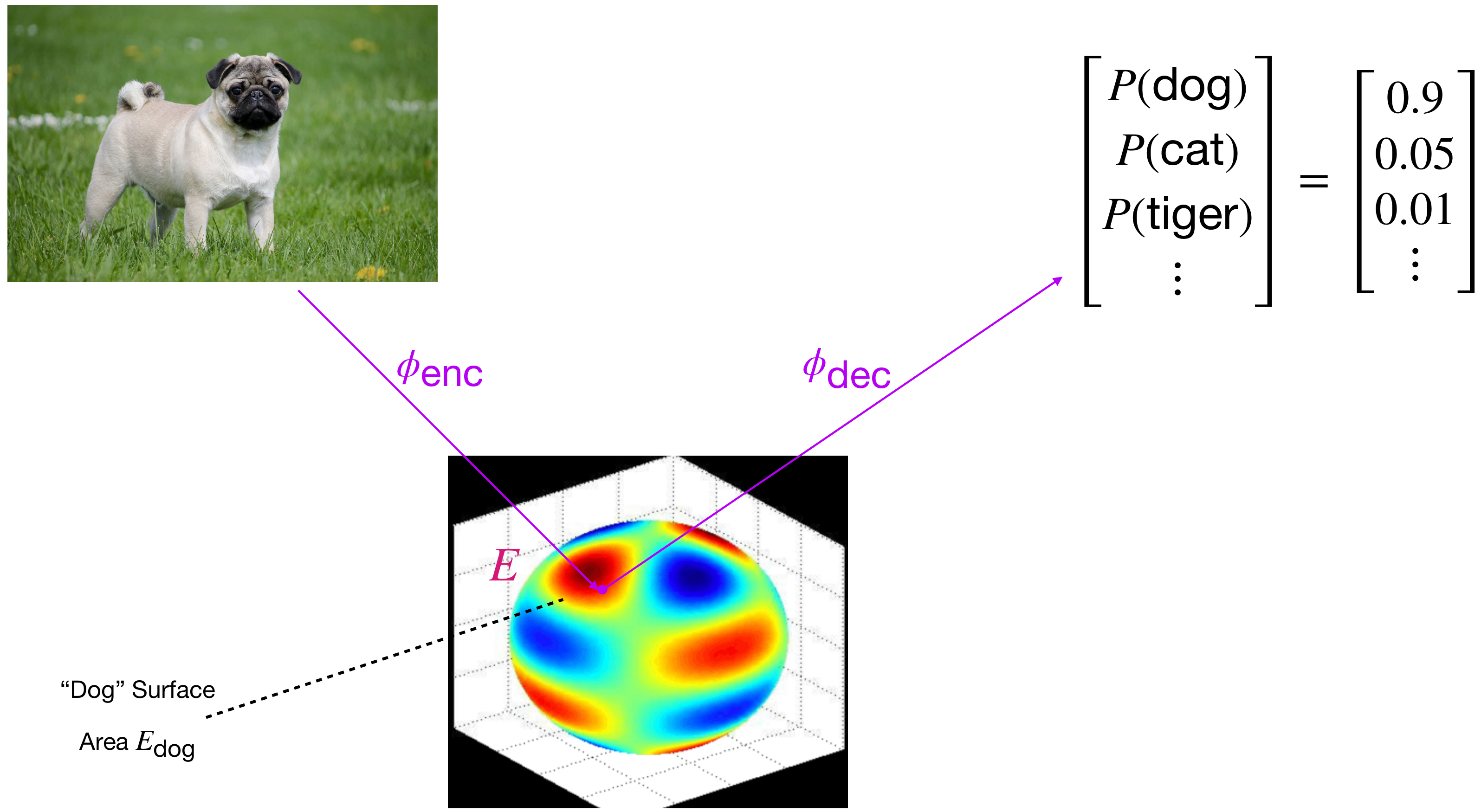}
\caption{In this figure, we provide a visual representation of the functioning of a convolutional autoencoder in an image classification problem. The input image of a pug is mapped by encoding function $\phi_\text{enc}$ to the surface area $E_\text{dog} \in\mathscr{E}$ corresponding to the feature ``dog''. Such a projection is then mapped by decoding function $\phi_\text{dec}$ to the probability space over the labels. In this simple example the model gives a high probability to the first label, i.e. ``dog''.}\label{IMSG_heuristic}
\end{figure}

Now, imagine that we want to study a smooth transition between a dog and a cat image. Intuitively, this corresponds to a ``walk'' from the ``dog portion'' $E_\text{dog}$ of the manifold's surface $\mathcal{S}(E)$ to the ``cat portion'' $E_\text{cat}$ \citep{perez}. A visual example is given in Figure \ref{IMSG_heuristic2}. Our results, and especially Corollaries \ref{limiting-cor} and \ref{cor:pointwise-erg-wass}, may be very useful when (i) such a ``walk'' is not deterministic, but rather modeled via a Markov process; (ii) the user is uncertain about which transition probabilities to choose (and on the nature of the invariant measure); and (iii) we are interested in the dynamics of the probability of a particular label during such a walk. For example, if we are interested in how the probability of the first label 
evolves over time, we could consider function $z \mapsto \tilde{f}(z)=e_1^\top \phi_\text{dec}(z)$, where $e_1=(1,0,0,\ldots,0)^\top$ is a $J$-dimensional standard basis vector, and $\tilde{f}$ is an example of a  bounded functional on $E$, as discussed earlier in this section.

\begin{figure}[h]%
\centering
\includegraphics[width=.55\textwidth]{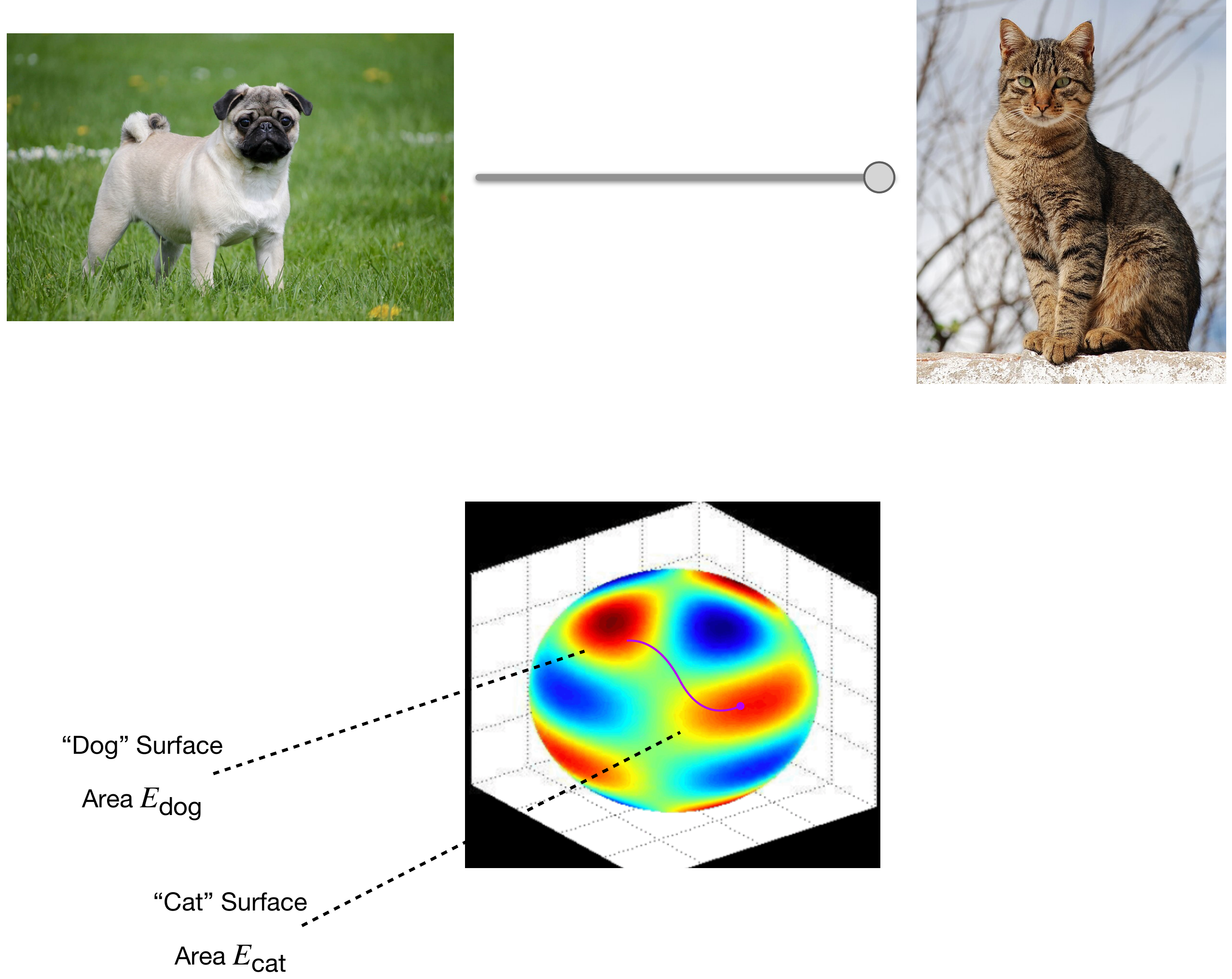}
\caption{In this figure, we represent the smooth transition between a dog and a cat pictures. Moving a cursor between the pictures causes a walk from the ``dog portion'' $E_\text{dog}$ of the manifold's surface $\mathcal{S}(E)$ to the ``cat portion'' $E_\text{cat}$, depicted as a purple smooth curve. The results we present in this paper are useful to determine the behavior of a similar walk that is not deterministic, when it is not possible to determine exactly transition probabilities (and the invariant measure, when it exists), and when we are interested e.g. in the evolution of the probability of the first label, $y_1=\text{``dog''}$, captured by the functional $z \mapsto \tilde{f}(z)=e_1^\top \phi_\text{dec}(z)$.}\label{IMSG_heuristic2}
\end{figure}

Our paper is structured as follows. Section \ref{prelim} introduces the concept of an Imprecise Markov Semigroup $\mathbf{Q}$. Section \ref{riemann} derives our results when state space $E$ is a Euclidean space or a Riemannian manifold, $\mathbf{Q}$ is compact, and the bounded test functional $\tilde f$ is fixed. Section \ref{idmt-section} generalizes the findings in Section \ref{riemann} to the case of $E$ being an arbitrary measurable space, while keeping $\mathbf{Q}$  compact and the bounded test functional $\tilde f$ fixed. Section \ref{subsec:wass-sublinear} further extends our results: It is only required that $E$ is Polish metric with finite diameter and that the test map is any bounded Lipschitz functional. Section \ref{concl} concludes our work. We prove Theorem \ref{thm-main-2} in Appendix \ref{app-a}, and we provide a primer on Markov semigroups and diffusion Markov triples in Appendix \ref{app-b}.


\section{Imprecise Markov Semigroups}\label{prelim}
In this section, we introduce the concept of an Imprecise Markov Semigroup. Its definition hinges on the notions of  Markov processes, Markov semigroups, carré du champ and iterated carré du champ operators, and Bakry-Émery curvature. We refer the unfamiliar reader to Appendix \ref{back-MSG}.


The following is introduced in \citet[Chapter 1.1]{bakry}. It is the minimal structure we require throughout the paper for the space where our stochastic processes of interest live.
\begin{definition}[Good Measurable Space \citep{bakry}]\label{good}
    A {\em good measurable space} is a measurable space $(E,\mathcal{F})$ for which the measure decomposition theorem applies, and for which there is a countable family of sets which generates the $\sigma$-algebra $\mathcal{F}$.
\end{definition}

Here, the measure decomposition theorem refers to the
\emph{disintegration theorem} (existence of regular conditional distributions).
Concretely, if $\mu$ is a probability measure on $(E\times E,\mathcal F\otimes\mathcal F)$
and $\mu_1$ is its first marginal, then there exists a Markov kernel $k(x,\text{d}y)$
(measurable in $x$) such that
\(
\mu(\text{d}x,\text{d}y) = k(x,\text{d}y) \mu_1(\text{d}x)
\). 
Equivalently, $\mu$ decomposes into the conditional laws of the second coordinate given the first. This is typically guaranteed on standard Borel spaces (in particular, Polish spaces with their Borel $\sigma$-algebras).

In this work, we are interested in Markov semigroups $\mathbf{P}=(P_t)_{t \geq 0}$ associated with (time homogeneous) Markov processes $(X_t)_{t \geq 0}$. For such a Markov semigroup, for every $t \geq 0$, operator $P_t$ is such that 

$$P_tf(x)\coloneqq \mathbb{E}\left[ f(X_t) \mid X_0=x \right] = \int_E f(y) p_t(x,\text{d}y), \quad f \in B(E) \text{, } x \in E,$$
where $B(E)$ is the space of bounded measurable functionals on $E$, and $p_t$ is the probability kernel describing the evolution of the Markov process. In words, operator $P_t$ applied on $f \in B(E)$ gives us the expected value of $f(X_t)$, ``weighted by'' the transition probability (from the starting value $X_0=x$ to $X_t$) of our Markov process. 
We focus on this type of Markov semigroups because studying their (ergodic) behavior informs us about the (ergodic) behavior of the associated Markov processes. 

Throughout the paper, we will assume that the semigroups we consider always have an invariant probability measure $\mu$, so that the elements of $\mathbf{P}=(P_t)_{t \geq 0}$ are bounded operators in $\mathbb{L}^p(\mu)$, $p \in [1,\infty)$. In fact, most semigroups of interest have an invariant measure \citep[Section 1.2.1]{bakry}. A way to derive $\mu$ as a weak limit of a ``reasonable'' initial probability measure $\mu_0$ is inspected in \citet[Page 10]{bakry}. For a discussion on this matter, we refer the interested reader to Remark \ref{spotlight} in Appendix \ref{back-MSG}.



We now introduce a space that will play a pivotal role in the definition of an Imprecise Markov Semigroup. Recall that, for all $t \geq 0$, the $t$-th element $P_t$ of Markov semigroup $\mathbf{P}=(P_t)_{t \geq 0}$ is a linear operator that sends bounded measurable functionals on $E$ to bounded measurable functionals on $E$. In formulas, we can write this as $P_t: B(E) \rightarrow B(E)$, or equivalently, $P_t \in B(E)^{B(E)}$. In turn, this implies that a Markov semigroup $\mathbf{P}$ can be seen as a function mapping bounded measurable functionals $f$ on $E$ to nets of bounded measurable functionals $\mathbf{P}f=(P_t f)_{t \geq 0}$ on $E$. In formulas, $\mathbf{P}: B(E) \rightarrow {B(E)}^{\mathbb{R}_+}$, or equivalently, $\mathbf{P} \in ({B(E)}^{\mathbb{R}_+})^{B(E)}$. 

\begin{definition}[Imprecise Markov Semigroup (IMSG)]\label{imsg-def-1}
    An Imprecise Markov Semigroup is a nonempty 
    collection $\mathbf{Q}$ of Markov semigroups (MSGs) associated with Markov processes. That is, $\mathbf{Q} \subset ({B(E)}^{\mathbb{R}_+})^{B(E)}$ such that
\begin{align}\label{imsg-first-property}
    \mathbf{Q} \subseteq \mathfrak{Q} \coloneqq \{\mathbf{P} \in ({B(E)}^{\mathbb{R}_+})^{B(E)} : \mathbf{P} \text{ is an MSG associated with a Markov process} \}.
\end{align}
\end{definition}

In Sections \ref{riemann} and \ref{idmt-section}, we will work with compact IMSGs. To do so, we need to introduce a notion of compactness in the space $({B(E)}^{\mathbb{R}_+})^{B(E)}$. Such a notion
hinges on the possibility of comparing the elements of $\mathbf{Q}$, when they are evaluated at some functional $\tilde{f}\in B(E)$.

First, suppose we are able to define a total preorder $\preceq^\text{tot}_{\tilde{f}}$ so that

$$\mathbf{P} \preceq^\text{tot}_{\tilde{f}} \mathbf{P^\prime} \iff P_t \tilde{f} \leq P^\prime_t \tilde{f} \text{, } \forall t \geq 0.$$
This relation is always reflexive and transitive on $\mathbf Q$; It need not be antisymmetric, since distinct semigroups may coincide on $\tilde f$.
We therefore introduce the equivalence relation
\[
\mathbf P \sim_{\tilde f} \mathbf P' \quad \Longleftrightarrow \quad
P_t\tilde f = P_t'\tilde f   \text{ for all } t\ge0,
\]
and work on the quotient set $\mathbf Q_{\tilde f}\coloneqq \mathbf Q/ \sim_{\tilde f}$.
The preorder $\preceq^\mathrm{tot}_{\tilde f}$ induces a genuine \emph{total order} on $\mathbf Q_{\tilde f}$ by
\[
[\mathbf P] \preceq^\mathrm{tot}_{\tilde f} [\mathbf P'] \quad \Longleftrightarrow \quad
\mathbf P \preceq^\mathrm{tot}_{\tilde f} \mathbf P',
\]
which is well-defined and antisymmetric by construction. Of course, here the square brackets denote equivalence classes. 
We have that, for all $[\mathbf{P}],[\mathbf{P^\prime}] \in \mathbf{Q}_{\tilde f}$, either
$[\mathbf{P}] \preceq^\text{tot}_{\tilde{f}} [\mathbf{P^\prime}]$ or
$[\mathbf{P^\prime}] \preceq^\text{tot}_{\tilde{f}} [\mathbf{P}]$.
In addition, the notions of maximal and greatest elements of $(\mathbf{Q}_{\tilde f},\preceq^\text{tot}_{\tilde{f}})$ coincide, and correspond to the classical notion of maximum (with respect to the total order $\preceq^\text{tot}_{\tilde{f}}$). Similarly, minimal and least element notions coincide, and correspond to the classical notion of minimum.

Now, let $\tau_{\preceq^\text{tot}_{\tilde{f}}}$ be the order topology on $\mathbf Q_{\tilde f}$. That is (provided $|\mathbf{Q}_{\tilde f}|\geq 2$), $\tau_{\preceq^\text{tot}_{\tilde{f}}}$ is the topology whose base is given by the open rays
$\{[\mathbf{P}] : [\mathbf{P}_1] \prec^\text{tot}_{\tilde{f}} [\mathbf{P}]\}$ and
$\{[\mathbf{P}] : [\mathbf{P}] \prec^\text{tot}_{\tilde{f}} [\mathbf{P}_2]\}$, 
where $\prec^\text{tot}_{\tilde{f}}$ is the strict counterpart of $\preceq^\text{tot}_{\tilde{f}}$,
for all $[\mathbf{P}_1]$, $[\mathbf{P}_2]$ in $\mathbf{Q}_{\tilde f}$,
and by the open intervals
$([\mathbf{P}_1],[\mathbf{P}_2]) = \{[\mathbf{P}] : [\mathbf{P}_1] \prec^\text{tot}_{\tilde{f}} [\mathbf{P}] \prec^\text{tot}_{\tilde{f}} [\mathbf{P}_2]\}$.
This means that the $\tau_{\preceq^\text{tot}_{\tilde{f}}}$-open sets in $\mathbf{Q}_{\tilde f}$ are the sets that are a union of (possibly infinitely many) such open intervals and rays.

Assume that $\mathbf{Q}_{\tilde f}$ is $\tau_{\preceq^\text{tot}_{\tilde{f}}}$-compact.
This implies that $\mathbf{Q}_{\tilde f}$ contains the minimum and the maximum (with respect to the total order $\preceq^\text{tot}_{\tilde{f}}$), which we denote by $\underline{[\mathbf{P}]}_{\tilde{f}}$ and $\overline{[\mathbf{P}]}_{\tilde{f}}$, respectively.
We fix arbitrary representatives $\underline{\mathbf{P}}_{\tilde{f}}\in \underline{[\mathbf{P}]}_{\tilde{f}}$ and $\overline{\mathbf{P}}_{\tilde{f}}\in \overline{[\mathbf{P}]}_{\tilde{f}}$.
In turn, since $\underline{\mathbf{P}}_{\tilde{f}},\overline{\mathbf{P}}_{\tilde{f}} \in \mathbf{Q}$, they both are (precise) Markov semigroups associated with Markov processes by \eqref{imsg-first-property}.

Suppose now that the induced relation is not total. More formally, the preorder $\preceq^\text{part}_{\tilde{f}}$ on $\mathbf{Q}$ is defined similarly to before by
\[
\mathbf{P} \preceq^\text{part}_{\tilde{f}} \mathbf{P^\prime}
\iff
P_t \tilde{f} \leq P^\prime_t \tilde{f} \text{, } \forall t \geq 0.
\]
Consider its induced relation (a genuine partial order) on the quotient $\mathbf Q_{\tilde f}$. Contrary to before, this time there might exist $[\mathbf{P}],[\mathbf{P^\prime}] \in \mathbf{Q}_{\tilde f}$ such that
$[\mathbf{P}] \not\preceq^\text{part}_{\tilde{f}} [\mathbf{P^\prime}]$ and
$[\mathbf{P}] \not\succeq^\text{part}_{\tilde{f}} [\mathbf{P^\prime}]$. A remark is in order.

\begin{remark}[Antisymmetry and Identifiability]
    Fix $\tilde f\in B(E)$ and consider the preorder on $\mathbf Q$ defined by
\[
\mathbf P\preceq_{\tilde f}\mathbf P'
\quad\Longleftrightarrow\quad
P_t\tilde f \le P_t'\tilde f   \text{, } \forall t\ge0.
\]
This preorder is \emph{antisymmetric} on $\mathbf Q$ (hence a genuine partial order on the semigroups in $\mathbf Q$) if and only if the orbit map
\[
\mathbf P \mapsto (P_t\tilde f)_{t\ge0}\in B(E)^{\mathbb R_+}
\]
is injective on $\mathbf Q$, i.e. whenever $P_t\tilde f=P_t'\tilde f$ for all $t\ge0$, it follows that $\mathbf P=\mathbf P'$ as semigroups.
Equivalently, $\tilde f$ \emph{separates} the elements of $\mathbf Q$ through their time evolution. In general, a single functional $\tilde f$ will not be separating, so assuming antisymmetry from $\tilde f$ can be a strong identifiability restriction on $\mathbf Q$. However, antisymmetry can hold in structured families (e.g. low-dimensional parametric classes) when $\tilde f$ is chosen so that the trajectories $t\mapsto P_t\tilde f$ uniquely identify the underlying semigroup within $\mathbf Q$. Identifiability of Markov kernel/semigroup can be guaranteed via a \emph{separating (measure-determining) class} of test functions rather than on a single $\tilde f$ \citep{ethier1986markov,blount2010separating}. In this paper we instead allow an arbitrary, application-driven choice of $\tilde f$ and handle potential non-antisymmetry by working with the quotient $\mathbf Q_{\tilde f}\coloneqq\mathbf Q/ \sim_{\tilde f}$. In the special case where $\tilde f$ is separating on $\mathbf Q$, the equivalence relation $\sim_{\tilde f}$ is trivial and the quotient poset becomes unnecessary.
\end{remark}

Let us now introduce three set theoretic concepts, inspired by \citet{wolk1958order}, that will allow us to properly define a compact Imprecise Markov Semigroup even when only a partial order can be determined.



\begin{definition}[Up/Down-Directed Subsets of $\mathbf{Q}_{\tilde f}$]\label{directed}
    A subset $\mathbf{S}$ of $\mathbf{Q}_{\tilde f}$ is {\em up-directed} if for all $[\mathbf{P}],[\mathbf{P^\prime}] \in \mathbf{S}$, there exists $[\mathbf{P^{\prime\prime}}]\in \mathbf{S}$ such that $[\mathbf{P}] \preceq^\text{part}_{\tilde{f}} [\mathbf{P^{\prime\prime}}]$ and $[\mathbf{P^\prime}] \preceq^\text{part}_{\tilde{f}} [\mathbf{P^{\prime\prime}}]$. We say that $\mathbf{S}$ is {\em down-directed} if the opposite holds. That is, if for all $[\mathbf{P}],[\mathbf{P^\prime}] \in \mathbf{S}$, there exists $[\mathbf{P^{\prime\prime}}]\in \mathbf{S}$ such that $[\mathbf{P^{\prime\prime}}] \preceq^\text{part}_{\tilde{f}} [\mathbf{P}]$ and $[\mathbf{P^{\prime\prime}}] \preceq^\text{part}_{\tilde{f}} [\mathbf{P^\prime}]$.
\end{definition}

\begin{definition}[Dedekind-Closed Subsets of $\mathbf{Q}_{\tilde f}$]\label{dedekind}
    A subset $\mathbf{K}$ of $\mathbf{Q}_{\tilde f}$ is {\em Dedekind-closed} if whenever $\mathbf{S}$ is an up-directed subset of $\mathbf{K}$, and $[\mathbf{P}]$ is the $\preceq^\text{part}_{\tilde{f}}$-least upper bound of $\mathbf{S}$ (or $\mathbf{S}$ is a down-directed subset of $\mathbf{K}$, and $[\mathbf{P}]$ is the $\preceq^\text{part}_{\tilde{f}}$-greatest lower bound of $\mathbf{S}$), then $[\mathbf{P}] \in \mathbf{K}$.
\end{definition}

\begin{definition}[Incomparability, Diversity, and $\preceq^\text{part}_{\tilde{f}}$-Width of $\mathbf{Q}_{\tilde f}$]\label{further-st-concepts}
    Two elements $[\mathbf{P}],[\mathbf{P^\prime}] \in \mathbf{Q}_{\tilde f}$ are {\em incomparable} if $[\mathbf{P}] \not\preceq^\text{part}_{\tilde{f}} [\mathbf{P^\prime}]$ and $[\mathbf{P}] \not\succeq^\text{part}_{\tilde{f}} [\mathbf{P^\prime}]$. A subset $\mathbf{S}$ of $\mathbf{Q}_{\tilde f}$ is {\em diverse} if $[\mathbf{P}],[\mathbf{P^\prime}] \in \mathbf{S}$ and $[\mathbf{P}]\neq[\mathbf{P^\prime}]$ implies that $[\mathbf{P}]$ and $[\mathbf{P^\prime}]$ are incomparable. Finally, we define the {\em $\preceq^\text{part}_{\tilde{f}}$-width} of $\mathbf{Q}_{\tilde f}$ to be the supremum of the set $\{k : k \text{ is the cardinality of a diverse subset of } \mathbf{Q}_{\tilde f}\}$. 
\end{definition}



Definitions \ref{directed}, \ref{dedekind}, and \ref{further-st-concepts} are needed to introduce the concept of $\preceq^\text{part}_{\tilde{f}}$-compatible topology. Requiring that $\mathbf{Q}_{\tilde{f}}$ is compact in such a topology (together with other assumptions) will ensure us that it possesses least and greatest elements (with respect to the partial order $\preceq^\text{part}_{\tilde{f}}$).

\begin{definition}[$\preceq^\text{part}_{\tilde{f}}$-Compatible Topology \citep{wolk1958order}]\label{order-comp-top}
    A topology $\tau_{\preceq^\text{part}_{\tilde{f}}}$ on $\mathbf{Q}_{\tilde f}$ is $\preceq^\text{part}_{\tilde{f}}$-compatible if
    \begin{itemize}
        \item every $\tau_{\preceq^\text{part}_{\tilde{f}}}$-closed set is also Dedekind-closed
        \item every set of the form $\{[\mathbf{P}]\in\mathbf{Q}_{\tilde f} : [\mathbf{P^\prime}] \preceq^\text{part}_{\tilde{f}} [\mathbf{P}] \preceq^\text{part}_{\tilde{f}} [\mathbf{P^{\prime\prime}}] \}$ is $\tau_{\preceq^\text{part}_{\tilde{f}}}$-closed.
    \end{itemize}
\end{definition}

The next remark explores the connections between the (generic) partial-order-compatible and the interval topologies.

\begin{remark}[Interval Topology vs $\preceq$-Compatible Topologies]
Let $(X,\preceq)$ be a partially ordered set. 
\citet{wolk1958order} considers two topologies naturally associated with the order.

\noindent\textbf{(1) Interval topology.}
The \emph{interval topology} $\mathcal T_I$ is the topology for which the family of order intervals
\[
[a,b]\coloneqq\{x\in X: a\preceq x\preceq b\}, \qquad a,b\in X,
\]
forms a \emph{closed subbasis}. Equivalently, $\mathcal T_I$ is the coarsest topology on $X$
for which every interval $[a,b]$ is closed.

\noindent\textbf{(2) $\preceq$-compatible topology.}
A topology $\mathcal T$ on $X$ is said to be \emph{$\preceq$-compatible} if:
(i) every $\mathcal T$-closed set is \emph{Dedekind-closed}, and
(ii) every interval $[a,b]$ is $\mathcal T$-closed.

\noindent\textbf{Relationship.}
Condition (ii) implies that any $\preceq$-compatible topology must contain the interval topology,
\[
\mathcal T_I \subseteq \mathcal T.
\]
Moreover, if $\mathcal T_D$ denotes the topology whose closed sets are exactly the Dedekind-closed
subsets of $X$, then condition (i) yields the upper bound
\[
\mathcal T \subseteq \mathcal T_D.
\]
Hence, $\preceq$-compatible topologies are precisely those lying between $\mathcal T_I$ and $\mathcal T_D$,
\[
\mathcal T_I \subseteq \mathcal T \subseteq \mathcal T_D.
\]
In particular, $\mathcal T_I$ is the \emph{minimal} $\preceq$-compatible topology requirement (it enforces
closed intervals only), whereas $\mathcal T_D$ is the \emph{maximal} one (it enforces all Dedekind-closed
sets to be closed). In special situations these two coincide (e.g. under finite $\preceq$-width, similarly to Lemma \ref{uniqueness}), yielding uniqueness of the $\preceq$-compatible topology.
\end{remark}

The following lemma summarizes some useful topological (and order-theoretic) properties of $\mathbf{Q}$.


\begin{lemma}[Topological Properties of $\mathbf{Q}_{\tilde f}$]\label{uniqueness}
    If $\mathbf{Q}_{\tilde f}$ has finite $\preceq^\text{part}_{\tilde{f}}$-width, then  $\preceq^\text{part}_{\tilde{f}}$-compatible topology $\tau_{\preceq^\text{part}_{\tilde{f}}}$ is unique. If in addition (a) $\mathbf Q_{\tilde f}$ is $\tau_{\preceq^\text{part}_{\tilde f}}$-compact, (b)
$\preceq^\text{part}_{\tilde f}$ has closed graph, that is, the set
\[
\mathrm{Graph}(\preceq^\text{part}_{\tilde f})
\coloneqq \{([\mathbf P],[\mathbf P^\prime])\in\mathbf Q_{\tilde f}\times\mathbf Q_{\tilde f}: [\mathbf P]\preceq^\text{part}_{\tilde f}[\mathbf P']\}
\]
is closed in the product topology $\tau_{\preceq^\text{part}_{\tilde f}}\times\tau_{\preceq^\text{part}_{\tilde f}}$, and (c)  $\mathbf Q_{\tilde f}$ is both up-  and down-  directed, then $\mathbf Q_{\tilde f}$ admits $\preceq^\text{part}_{\tilde f}$-least and $\preceq^\text{part}_{\tilde f}$-greatest elements. We denote them by $\underline{[\mathbf{P}]}_{\tilde{f}}$ and $\overline{[\mathbf{P}]}_{\tilde{f}}$, respectively.
\end{lemma}


\begin{proof}
    The first part of the lemma is a consequence of \citet[Theorem, page 528]{wolk1958order}, which also shows that $(\mathbf{Q}_{\tilde f},\tau_{\preceq^\text{part}_{\tilde{f}}})$ is Hausdorff. 
    
    For the second claim, assume that $\mathbf Q_{\tilde f}$ is $\tau_{\preceq^\text{part}_{\tilde f}}$-compact and that
$\preceq^\text{part}_{\tilde f}$ has closed graph. 
Consider the family
\[
\mathcal C \coloneqq \{\downarrow [\mathbf P] : [\mathbf P]\in\mathbf Q_{\tilde f}\},
\qquad
\downarrow [\mathbf P] \coloneqq \{[\mathbf R]\in\mathbf Q_{\tilde f}:[\mathbf R]\preceq^\text{part}_{\tilde f}[\mathbf P]\}.
\]
Each $\downarrow [\mathbf P]$ is nonempty and closed: Indeed,
\[
\downarrow [\mathbf P]=\{[\mathbf R]\in\mathbf Q_{\tilde f}:([\mathbf R],[\mathbf P])\in \mathrm{Graph}(\preceq^\text{part}_{\tilde f})\},
\]
so it is the preimage of the closed set $\mathrm{Graph}(\preceq^\text{part}_{\tilde f})$ under the continuous map
$[\mathbf R]\mapsto([\mathbf R],[\mathbf P])$.
Order $\mathcal C$ by reverse inclusion.
Let $\mathcal K\subset\mathcal C$ be a chain for this order, i.e. a nested family of nonempty closed sets.
By compactness, $\bigcap_{C\in\mathcal K} C\neq\emptyset$; Pick $[\mathbf P]\in\bigcap_{C\in\mathcal K} C$.
Fix $C\in\mathcal K$. Since $C=\downarrow [\mathbf Q_0]$ for some $[\mathbf Q_0]\in\mathbf Q_{\tilde f}$ and $[\mathbf P]\in C$, we have $[\mathbf P]\preceq^\text{part}_{\tilde f}[\mathbf Q_0]$.
If $[\mathbf R]\preceq^\text{part}_{\tilde f}[\mathbf P]$, then by transitivity $[\mathbf R]\preceq^\text{part}_{\tilde f}[\mathbf Q_0]$, hence $[\mathbf R]\in C$.
Thus $\downarrow [\mathbf P]\subseteq C$.
Hence $\downarrow [\mathbf P]$ is an upper bound of $\mathcal K$
in the reverse-inclusion order. Therefore, by Zorn's lemma, $\mathcal C$ has a maximal element $\downarrow [\mathbf P_\ast]$
(with respect to reverse inclusion).
We claim that $[\mathbf P_\ast]$ is $\preceq^\text{part}_{\tilde f}$-minimal in $\mathbf Q_{\tilde f}$.
Indeed, if $[\mathbf R]\preceq^\text{part}_{\tilde f}[\mathbf P_\ast]$, then $\downarrow [\mathbf R]\subseteq \downarrow [\mathbf P_\ast]$.
By maximality of $\downarrow [\mathbf P_\ast]$ in $\mathcal C$ (reverse inclusion), we must have
$\downarrow [\mathbf R]=\downarrow [\mathbf P_\ast]$, which implies $[\mathbf P_\ast]\preceq^\text{part}_{\tilde f}[\mathbf R]$.
By antisymmetry of $\preceq^\text{part}_{\tilde f}$ on $\mathbf Q_{\tilde f}$, $[\mathbf R]=[\mathbf P_\ast]$. Hence $[\mathbf P_\ast]$ is minimal.
The existence of a maximal element follows analogously by applying the same argument to the closed principal upper sets
$\uparrow [\mathbf P] \coloneqq \{[\mathbf R]\in\mathbf Q_{\tilde f}:[\mathbf P]\preceq^\text{part}_{\tilde f}[\mathbf R]\}$. 
Finally, assume $\mathbf Q_{\tilde f}$ is down-directed. Let $[\mathbf P_\ast]$ be a $\preceq^\text{part}_{\tilde f}$-minimal element.
For any $[\mathbf P]\in\mathbf Q_{\tilde f}$, by down-directedness there exists $[\mathbf R]\in\mathbf Q_{\tilde f}$ with
$[\mathbf R]\preceq^\text{part}_{\tilde f}[\mathbf P_\ast]$ and $[\mathbf R]\preceq^\text{part}_{\tilde f}[\mathbf P]$.
By minimality of $[\mathbf P_\ast]$ we must have $[\mathbf R]=[\mathbf P_\ast]$, hence
$[\mathbf P_\ast]\preceq^\text{part}_{\tilde f}[\mathbf P]$. Thus $[\mathbf P_\ast]$ is the least element.
The argument for a greatest element is analogous using up-directedness.
\end{proof}



Let us add a discussion on why the conditions in Lemma~\ref{uniqueness} ensure that the greatest and least elements for $\mathbf Q_{\tilde f}$ exist.
Compactness with respect to the topology $\tau_{\preceq^\text{part}_{\tilde f}}$ can be read as a ``no escape'' condition:
Whenever we approximate an extremal element by moving along a chain,
the corresponding limit points -- when they exist -- remain inside $\mathbf Q_{\tilde f}$.
In other words, $\tau_{\preceq^\text{part}_{\tilde f}}$-compactness ensures that $\mathbf Q_{\tilde f}$ does not ``escape to infinity''
in a topological sense: It prevents extremal candidates from slipping out of the model class.

Finite $\preceq^\text{part}_{\tilde f}$-width restricts the degree of ``spread'' in the poset $\mathbf Q_{\tilde f}$,
preventing it from having arbitrarily large levels of incomparability. In other words, it ensures that the ``pathological elements''
of $\mathbf Q_{\tilde f}$ -- i.e. the incomparable ones -- are controllably few. This helps bridge the gap with the totally ordered case
$\preceq^\text{tot}_{\tilde f}$, where no incomparable elements exist and extremal elements are conceptually straightforward.

The closed-graph condition for $\preceq^\text{part}_{\tilde f}$ (on the quotient poset $\mathbf Q_{\tilde f}$) plays a complementary topological role.
It says that comparability is stable under limits: If $[\mathbf P_n]\preceq^\text{part}_{\tilde f}[\mathbf P_n']$ for all $n$
and $[\mathbf P_n]\to[\mathbf P]$, $[\mathbf P_n']\to[\mathbf P']$ in $\tau_{\preceq^\text{part}_{\tilde f}}$, then necessarily
$[\mathbf P]\preceq^\text{part}_{\tilde f}[\mathbf P']$. 
In particular, each principal lower set
$\downarrow[\mathbf P]\coloneqq\{[\mathbf R]\in\mathbf Q_{\tilde f}:[\mathbf R]\preceq^\text{part}_{\tilde f}[\mathbf P]\}$ and principal upper set
$\uparrow[\mathbf P]\coloneqq\{[\mathbf R]\in\mathbf Q_{\tilde f}:[\mathbf P]\preceq^\text{part}_{\tilde f}[\mathbf R]\}$ is $\tau_{\preceq^\text{part}_{\tilde f}}$-closed.
This prevents the order relation from ``breaking'' at the boundary of $\mathbf Q_{\tilde f}$ and is what allows compactness
arguments (via intersections of closed sets) to produce genuine extremal elements \emph{inside} $\mathbf Q_{\tilde f}$.

Finally, the up-directed and down-directed assumptions ensure that these extremal elements are not merely maximal/minimal
in the partial order, but are in fact greatest/least. Down-directedness means that any two models admit a common lower bound in $\mathbf Q_{\tilde f}$;
Thus, once a $\preceq^\text{part}_{\tilde f}$-minimal element $[\underline{\mathbf P}]_{\tilde f}$ exists, it must lie below every other element
(otherwise one could find a strict lower bound for it), hence it is the \emph{least} element. The analogous statement holds for the greatest
element under up-directedness. In short, compactness and closed graph yield existence of minimal/maximal elements, while two-sided directedness
upgrades them to least/greatest, which is precisely what we need to identify the extremal elements within $\mathbf Q_{\tilde f}$.

We now give a simple example of a collection $\mathbf{Q}$ and a test function $\tilde f$ that satisfy \eqref{imsg-first-property} and the hypotheses of Lemma \ref{uniqueness}, without the additional complications arising from equivalence classes.

\begin{example}[A Compact Family with Closed Order and Least/Greatest Elements]\label{ex:compact-chain}
Let $E=\{0,1\}$ with the discrete $\sigma$-algebra $\mathcal F=2^E$, and fix the test function
$\tilde f:E\to\mathbb R$ given by $\tilde f(x)=x$.
For each parameter $a\in[0,1]$, let $\mu_a$ be the Bernoulli$(a)$ law on $E$, i.e.
\[
\mu_a(\{1\})=a,\qquad \mu_a(\{0\})=1-a.
\]
Define the Markov operator $K_a$ by
\[
(K_a g)(x)\coloneqq \int_E g(y) \mu_a(\mathrm dy),\qquad g\in B(E),
\]
so that $K_a g$ is constant in $x$. Notice that $K_a$ is idempotent, $K_a^2=K_a$.

For $t\ge0$, define
\[
P_t^{a}\coloneqq e^{-t}I+(1-e^{-t})K_a.
\]
Then $\mathbf P^{a}=(P_t^{a})_{t\ge0}$ is a Markov semigroup. Indeed, using $K_a^2=K_a$ and $IK_a=K_aI=K_a$,
\[
P_t^{a}P_s^{a}
=\bigl(e^{-t}I+(1-e^{-t})K_a\bigr)\bigl(e^{-s}I+(1-e^{-s})K_a\bigr)
=e^{-(t+s)}I+\bigl(1-e^{-(t+s)}\bigr)K_a
=P_{t+s}^{a}.
\]

Let
\[
\mathbf Q\coloneqq \{\mathbf P^{a}: a\in[0,1]\}.
\]
With respect to the order $\preceq^{\mathrm{part}}_{\tilde f}$ (i.e. $\mathbf P\preceq^{\mathrm{part}}_{\tilde f}\mathbf P^\prime
\iff P_t\tilde f\le P_t^\prime\tilde f$ for all $t\ge0$), we have
\[
P_t^{a}\tilde f(x)
=e^{-t}\tilde f(x)+(1-e^{-t})\mu_a(\tilde f)
=e^{-t}x+(1-e^{-t})a,
\]
and therefore
\[
\mathbf P^{a}\preceq^{\mathrm{part}}_{\tilde f}\mathbf P^{b}\quad \Longleftrightarrow\quad a\le b.
\]
Hence $\mathbf Q$ is a chain (so it has finite width $1$).

Endow $\mathbf Q$ with the topology induced by the identification $a\mapsto \mathbf P^{a}$.
Then $\mathbf Q$ is compact and Hausdorff (homeomorphic to $[0,1]$), and the order graph is closed since it corresponds to
\[
\{(a,b)\in[0,1]^2: a\le b\},
\]
which is closed in $[0,1]^2$. In particular, every order interval
$\{\mathbf P\in\mathbf Q: \mathbf P^{a^\prime}\preceq^{\mathrm{part}}_{\tilde f}\mathbf P\preceq^{\mathrm{part}}_{\tilde f}\mathbf P^{a^{\prime\prime}}\}$
is closed.

Finally, $\mathbf Q$ admits a least and a greatest element:
\[
\underline{\mathbf P}_{\tilde f}=\mathbf P^{0},
\qquad
\overline{\mathbf P}_{\tilde f}=\mathbf P^{1}.
\]
\end{example}

We are now ready for the full definition of a compact Imprecise Markov Semigroup.

\begin{definition}[Compact Imprecise Markov Semigroup]\label{cpct-imsg-def}
    Let $\mathbf{Q}$ be an IMSG, that is, a subset of the space $ ({B(E)}^{\mathbb{R}_+})^{B(E)}$ that satisfies \eqref{imsg-first-property}. Fix any $\tilde{f}\in B(E)$, and recall the quotient poset $\mathbf Q_{\tilde f}=\mathbf Q/ \sim_{\tilde f}$ introduced above. If the induced order $\preceq^\text{tot}_{\tilde{f}}$ on $\mathbf Q_{\tilde f}$ is total, then we say that $\mathbf{Q}$ is a {\em compact  Imprecise Markov Semigroup} if $\mathbf Q_{\tilde f}$ is $\tau_{\preceq^\text{tot}_{\tilde{f}}}$-compact. 
    
    If instead the induced order $\preceq^\text{part}_{\tilde{f}}$ on $\mathbf Q_{\tilde f}$ is only partial, then we say that $\mathbf{Q}$ is a {\em compact Imprecise Markov Semigroup} if $\mathbf Q_{\tilde f}$ has finite $\preceq^\text{part}_{\tilde{f}}$-width, it is $\tau_{\preceq^\text{part}_{\tilde{f}}}$-compact, it is both up- and down-directed, and 
$\preceq^\text{part}_{\tilde f}$ has closed graph.
    
    In both cases, we denote by $\underline{[\mathbf{P}
    ]}_{\tilde{f}}$ and $\overline{[\mathbf{P}]}_{\tilde{f}}$ the $\preceq_{\tilde f}$-least and $\preceq_{\tilde f}$-greatest elements of $\mathbf Q_{\tilde f}$, respectively.
\end{definition}




\begin{remark}[Notation: Quotient Poset, Representatives, and Shorthands]\label{quotient-notation}
Fix $\tilde f\in B(E)$. Whenever we invoke order-theoretic or topological notions associated with $\preceq_{\tilde f}$, we work on the quotient poset $\mathbf Q_{\tilde f}=\mathbf Q/ \sim_{\tilde f}$ endowed with the induced antisymmetric order, and we use bracket notation to emphasise equivalence classes. In particular, $\underline{[\mathbf{P}]}_{\tilde{f}}$ and $\overline{[\mathbf{P}]}_{\tilde{f}}$ denote the $\preceq_{\tilde f}$-least and $\preceq_{\tilde f}$-greatest \emph{equivalence classes} in $\mathbf Q_{\tilde f}$. 

When we drop the brackets and write $\underline{\mathbf{P}}_{\tilde{f}}$ or $\overline{\mathbf{P}}_{\tilde{f}}$, we mean that we have fixed (once and for all in the relevant argument) a representative Markov semigroup in the corresponding class, i.e. $\underline{\mathbf{P}}_{\tilde{f}}\in\underline{[\mathbf{P}]}_{\tilde{f}}$ and $\overline{\mathbf{P}}_{\tilde{f}}\in\overline{[\mathbf{P}]}_{\tilde{f}}$. Any operator-level statement involving $\underline{\mathbf{P}}_{\tilde{f}}$ or $\overline{\mathbf{P}}_{\tilde{f}}$ (e.g. generators, invariant measures, curvature conditions) is to be understood as a statement about the chosen representatives.
For notational simplicity, in the remainder of this section, and in Sections \ref{riemann} and \ref{idmt-section}, we adopt the following abuse of notation,
$\underline{\mathbf{P}}_{\tilde{f}} \equiv \underline{\mathbf{P}}$ and $\overline{\mathbf{P}}_{\tilde{f}} \equiv \overline{\mathbf{P}}$, with the understanding that the dependence on the chosen function $\tilde{f}$ is implicit.
\end{remark}

At this point, the reader may ask themselves why we went to such great lengths to ensure that $\mathbf{Q}$ possesses the least and greatest elements $\underline{\mathbf{P}}$ and $\overline{\mathbf{P}}$, respectively.
Informally, the reason is that studying their limiting behavior (that is, studying the limits $\lim_{t \rightarrow\infty} \underline{P}_t \tilde{f}$ and $\lim_{t \rightarrow\infty} \overline{P}_t \tilde{f}$) will inform us on the limiting behavior of all the elements of $\mathbf{Q}$ evaluated at $\tilde{f}$. This intuition is made formal in Corollaries \ref{limiting-cor} and \ref{limiting-cor2}.
In addition, since $\underline{\mathbf{P}}$ and $\overline{\mathbf{P}}$ are (precise) Markov semigroups, we can use techniques developed in the Markov diffusion operators literature \citep{bakry} to carry out such a study.


Notice also that the framework of Definition~\ref{cpct-imsg-def} is deliberately \emph{non-switching}. 
That is, we model ambiguity by a \emph{set of precise Markov semigroups} \(\mathbf Q\), but we do not allow the decision-maker to concatenate two different models over successive time intervals. Concretely, given \(\mathbf P,\mathbf P^\prime\in\mathbf Q\), we do \emph{not} assume that there exists \(\mathbf P^{\prime\prime}\in\mathbf Q\) such that
\[
P^{\prime\prime}_{t+s}=P^\prime_t\circ P_s,\quad \text{ for some }t,s\ge0,
\]
nor do we interpret \(P^\prime_t\) as acting on the random output \(P_s\tilde f\) produced under a different model \(\mathbf P\).
Allowing such concatenations corresponds to admitting \emph{switching} controls, and typically leads to a time-consistent upper (or lower) expectation that is no longer representable by a single precise semigroup: The resulting one-step kernels need not be dominated by a fixed probability kernel and may instead be described by nonlinear transition operators -- e.g. capacities in the sense of \citet{choquet}. 

In Sections~\ref{riemann} and \ref{idmt-section} we therefore restrict attention to ambiguity represented by a collection of \emph{precise} semigroups without switching. By contrast, in Section~\ref{subsec:wass-sublinear} we explicitly \emph{allow} switching and construct the associated (Nisio) nonlinear semigroups, which lets us dispense with compactness while retaining robust long-term behavior.


Furthermore, it is worth mentioning that in Sections \ref{riemann} and \ref{idmt-section} we focus our attention on
orders (total and partial) associated with a particular function $\tilde{f}$. Indeed, if one were to consider the
``global'' comparison
\[
\mathbf{P} \preceq \mathbf{P}' \quad \iff \quad P_t f \le P_t^\prime f, \text{ for all } t\ge0 \text{ and all } f\in B(E),
\]
then this relation is in fact \emph{trivial}: Applying it to $-f$ yields $P_t f \ge P_t^\prime f$, hence $P_t f = P_t^\prime f$ for all
$t\ge0$ and $f\in B(E)$. In other words, such an order collapses to equality and cannot distinguish distinct semigroups.
Therefore, to obtain a nontrivial comparison and meaningful extremal elements, one must restrict the function class,
which we do here by working with orders induced by a fixed $\tilde f$.
We note that even under a restricted order, there are families $\mathbf Q$ satisfying
\eqref{imsg-first-property} that still fail to possess a least element $\underline{\mathbf{P}}$ with respect to that order, see \citet[Example~6.2]{krak}.

As we can see, the choice of $\tilde{f}$ is then crucial, and it should be driven by the application of interest. For example, as we discussed in Section \ref{intro}, in the case of a convolutional autoencoder, we can pick $\tilde{f}(z)=e_1^\top \phi_\text{dec}(z)$ to study how the probability of the first label evolves over time. In quantitative finance, we could pick 
        \(
          \tilde f(z)=\min\{(K - z_i)_+,M\},
        \)
        $z\in\mathbb{R}_+^{d}$, where \((\cdot)_+=\max(\cdot,0)\), \(K\) is the strike, and \(M\) a
        truncation level. This is a truncated put‑option payoff on the
        \(i\)-th asset; The resulting lower and upper semigroups yield worst‑case option prices under model uncertainty. In reliability/survival analysis, a viable choice is
        \(
          \tilde f(z)=\text{Ind}_{\{z\le\theta\}},
        \)
        for a failure threshold \(\theta\), where $\text{Ind}_{\{\cdot\}}$ denotes the indicator function. The functional tracks the probability that the
        degradation index \(X_t\) remains in the ``healthy'' region, yielding
        robust upper and lower survival curves.

    Let us now link the least element 
    $\underline{\mathbf{P}}$ with lower previsions \citep{decooman}, sometimes also referred to as lower expectations, a crucial concept -- perhaps the most important one -- in imprecise probability theory.\footnote{A similar argument holds for $\overline{\mathbf{P}}$ and upper previsions.} Given a generic set $\mathcal{P}$ of countably additive probability measures on $E$, the lower prevision for a (measurable bounded) functional $\tilde{f}$ on $E$ is given by $\underline{\mathbb{E}}(\tilde{f}) = \inf_{P\in\mathcal{P}} \mathbb{E}_P(\tilde{f})$, 
    whenever the expectations exist. Alternatively, a superadditive 
    Choquet capacity $\underline{P}$ can be considered,  
    and the lower 
    prevision is computed according to the Choquet integral $\int_E \tilde{f} \text{d}\underline{P}$ 
    \citep{choquet}. These two definitions only coincide if $\underline{P}$ is $2$-monotone \citep[Section 2.1.(ii)]{cerreia}. Conditional lower 
    previsions are defined analogously \citep{BARTL,decooman}. In our imprecise Markov semigroup framework of Sections \ref{riemann} and \ref{idmt-section}, we have that for every $t\geq 0$, the lower conditional prevision $\underline{\mathbb{E}}(\tilde{f}(X_t) \mid X_0)$ is equal to $\underline{P}_t \tilde{f}$, which in turn is a linear conditional expectation. This is because $\underline{\mathbf{P}}$ is the least element of $\mathbf{Q}$, so $\underline{\mathbf{P}}\in \mathbf{Q}$, which by \eqref{imsg-first-property} implies that $\underline{\mathbf{P}}$ is a (precise) Markov semigroup associated with a (linear) Markov process.
    Working with compact Imprecise Markov Semigroups (IMSGs) lets us represent, within a single object, the agent's ambiguity about both (i) the long-run law (invariant measure) and (ii) the short-run dynamics (transition probabilities) of a Markov process.
In Section \ref{subsec:wass-sublinear}, instead, we do not require $\mathbf{Q}$ to be compact, so we are able to explicitly relate imprecise Markov semigroups to both the lower envelope $\underline{\mathbb E}(\tilde f)$ and the Choquet integral $\int_E \tilde f \text{d}\underline{P}$.

In the imprecise probability literature, ambiguity about invariant behavior is often encoded directly at the level of stationary objects, for instance through an invariant Choquet capacity  \citep{de_Cooman_Hermans_Quaeghebeur_2009}. Ambiguity about the dynamics is instead typically modeled by sets of transition probabilities (in discrete time) or sets of transition rate matrices (in continuous time), frequently assumed convex so that worst/best-case evaluations can be computed via linear-programming-type procedures \citep{krak,hitting}.

A Markov semigroup does not, by itself, specify an initial distribution: It encodes only the transition mechanism
\((P_t)_{t\ge0}\), while an initial law \(\nu\in\Delta_E\) must be supplied separately to determine the full process law; Here $\Delta_E$ denotes the space of countably additive probability measures on $E$.
In this paper, we represent ambiguity about the \emph{dynamics} by an imprecise Markov semigroup \(\mathbf Q\), i.e. a family
of precise semigroups \(\mathbf P=(P_t)_{t\ge0}\). Each \(\mathbf P\in\mathbf Q\) determines a candidate collection of transition
probabilities and (when it exists) an associated invariant measure \(\mu \equiv \mu_{\mathbf P}\). Hence \(\mathbf Q\) captures ambiguity
both about transition probabilities and about invariant/long-run behavior, by allowing these objects to vary across
\(\mathbf P\in\mathbf Q\).
If one also wishes to model ambiguity about the initial distribution, this can be done exogenously by specifying a set of
priors \(\mathsf N\subseteq\Delta_E\) and considering \(\{\nu P_t:\nu\in\mathsf N, \mathbf P\in\mathbf Q\}\); Our results
primarily focus on the robust long-term behavior induced by uncertainty in the semigroup itself.

As we shall see in Sections \ref{riemann}, \ref{idmt-section}, and \ref{subsec:wass-sublinear}, taking an (imprecise Markov) semigroup route allows us to explicitly account for the geometry of the state space $E$ when we study the ergodic behavior of $\overline{\mathbf{P}}$ and $\underline{\mathbf{P}}$ (a perspective that is new to the literature), and to work in continuous time (whereas the majority of the results in the literature pertain the discrete time case). The price we pay in Sections \ref{riemann} and \ref{idmt-section} for these gains is the need for compactness of the equivalence class space under our chosen preorder.
This is because, as a consequence of Lemma \ref{uniqueness}, $\overline{\mathbf{P}}, \underline{\mathbf{P}} \in \mathbf{Q}$, and so we can use techniques developed for (precise) Markov semigroups to study the ergodic behavior of $\overline{\mathbf{P}}$ and $\underline{\mathbf{P}}$, and in turn of all the elements of $\mathbf{Q}$. In Section \ref{subsec:wass-sublinear}, instead, we allow $\mathbf{Q}$ to be non-compact, thus making our approach closer to
the frameworks introduced e.g. in 
\citet{CRIENS2024104354, denk2, DENK, erreygers23a, feng-zhao,kuhn,nendel},\footnote{In those papers, the curvature of the state space $E$ is not considered by the authors when deriving their results.} where nonlinear Markov semigroups (and in particular Markov semigroups associated with nonlinear Markov processes) are studied. 

\section{Ergodicity in the Riemannian Manifold Case -- Compact $\mathbf{Q}$}\label{riemann}
In this section, we begin our inspection on how the geometry of the state space $E$ informs us around the ergodicity of an imprecise Markov semigroup $\mathbf{Q}$. Our journey starts by letting $E=\mathbb{R}^n$, or $E=M$, a complete (with respect to the Riemannian distance) Riemannian manifold \citep{panoramic}.\footnote{For a primer on Riemannian geometry, we refer the reader to \citet[Appendix C.3]{bakry}} 

Loosely, a complete Riemannian manifold $M$ can be thought of as a ``smooth'' 
subset of $\mathbb{R}^n$, equipped with a complete metric that depends in some sense on the elements $x \in M$. A visual representation of a Riemannian manifold in $\mathbb{R}^3$ is given in Figure \ref{fig-riem}. In this context, the carré du champ and the iterated carré du champ operators associated with a Markov semigroup are easily computed (see Appendix \ref{riemann-manif}). Recall from Remark \ref{quotient-notation} that, in the following theorem, we have 
$\underline{\mathbf{P}} \equiv \underline{\mathbf{P}}_{\tilde{f}}$, a fixed chosen representative from $\underline{[\mathbf{P}]}_{\tilde{f}}$, and similarly for $\overline{\mathbf{P}}$.


\begin{theorem}[Limiting Behavior of $\underline{\mathbf{P}}$ and $\overline{\mathbf{P}}$, $E$  Euclidean space or Riemannian manifold]\label{thm-main-1}
    Let $E$ be $\mathbb{R}^n$ or a smooth complete connected manifold. Let $\mathcal{A}_0$ be the class of smooth ($\mathcal{C}^\infty$) compactly supported functionals $f$ on $E$.\footnote{Recall that a smooth function on a closed and bounded set (and hence, in this framework, compact by Heine-Borel, or Hopf-Rinow for Riemannian manifold) is itself bounded.} Fix any $\tilde{f}\in B(E)$, and assume that we can define the partial order $\preceq_{\tilde{f}}^\text{part}$ or the total order $\preceq_{\tilde{f}}^\text{tot}$ that we introduced in Section \ref{prelim}. Let $\mathbf{Q}$ be a compact imprecise Markov semigroup. Denote by $\overline{\mathbf{P}}$ and $\underline{\mathbf{P}}$ the greatest and least elements of $\mathbf{Q}$,\footnote{With respect to $\preceq^\text{part}_{\tilde{f}}$ or $\preceq^\text{tot}_{\tilde{f}}$.} respectively, with invariant measures $\overline{\mu}$ and $\underline{\mu}$, and associated infinitesimal generators $\overline{L}$ and $\underline{L}$, respectively.
    Suppose that the following three conditions hold,
    \begin{itemize}
        \item[(i)] $\underline{L}$ and $\overline{L}$ are elliptic diffusion operators, symmetric with respect to $\underline{\mu}$ and $\overline{\mu}$,\footnote{The symmetric assumption tell us that the law of the Markov processes associated with $\underline{\mathbf{P}}$ and $\overline{\mathbf{P}}$ are “reversible in time”.} respectively;
        \item[(ii)] $\underline{L}$ and $\overline{L}$ satisfy the Bakry-Émery curvature conditions $CD(\rho_{\underline{L}},\infty)$ and $CD(\rho_{\overline{L}},\infty)$, respectively, for some $\rho_{\underline{L}},\rho_{\overline{L}} \in \mathbb{R}$;
        \item[(iii)] The carré du champ algebra for both the carré du champ operators $\underline{\Gamma}$ and $\overline{\Gamma}$ of the Markov generators $\underline{L}$ and $\overline{L}$, respectively, is $\mathcal{A}_0$.
    \end{itemize}
    Then, we have that
    \begin{equation*}
        \forall f \in \mathbb{L}^2(\underline{\mu}), \quad  \lim_{t \rightarrow \infty} \underline{P}_t f = \mathbb{E}_{\underline{\mu}}\left[ f \right] =\int_E f \text{d}\underline{\mu}, \quad \underline{\mu}\text{-a.e.}
    \end{equation*}
    and 
    \begin{equation*}
    \forall f \in \mathbb{L}^2(\overline{\mu} ), \quad
        \lim_{t \rightarrow \infty} \overline{P}_t  f = \mathbb{E}_{\overline{\mu} }\left[ f \right] =\int_E f \text{d}\overline{\mu} , \quad \overline{\mu} \text{-a.e.}
    \end{equation*}
    in $\mathbb{L}^2(\underline{\mu} )$ and in $\mathbb{L}^2(\overline{\mu} )$, respectively. In addition, if the invariant measure for $\underline{\mathbf{P}} $ is equal to that of $\overline{\mathbf{P}} $, and if we denote it by $\mu $, then 
    \begin{equation}\label{ergodic-equiv-imp}
    \forall f \in \mathbb{L}^2(\mu ), \quad
        \lim_{t \rightarrow \infty} \overline{P}_t  f = \lim_{t \rightarrow \infty} \underline{P}_t  f = \mathbb{E}_{\mu }\left[ f \right] =\int_E f \text{d}\mu , \quad \mu \text{-a.e.}
    \end{equation}
    in $\mathbb{L}^2(\mu )$.
\end{theorem}

If $E$ is a complete Riemannian manifold, the Bakry-Émery curvature condition $CD(\rho_{\underline{L}},\infty)$ in Theorem \ref{thm-main-1}.(ii) is related to the notion of Ricci tensor at every point of $E$, and in turn to the Ricci curvature of the manifold $E$ (for more, we refer the reader to Appendix \ref{riemann-manif}).\footnote{A similar argument holds for the infinitesimal generator $\overline{L}$ associated with Markov semigroup $\overline{\mathbf{P}}$.} 
We see, then, how
the geometry of the state space $E$ influences the ergodic behavior of an imprecise Markov semigroup.

\begin{figure}[h]%
\centering
\includegraphics[width=0.5\textwidth]{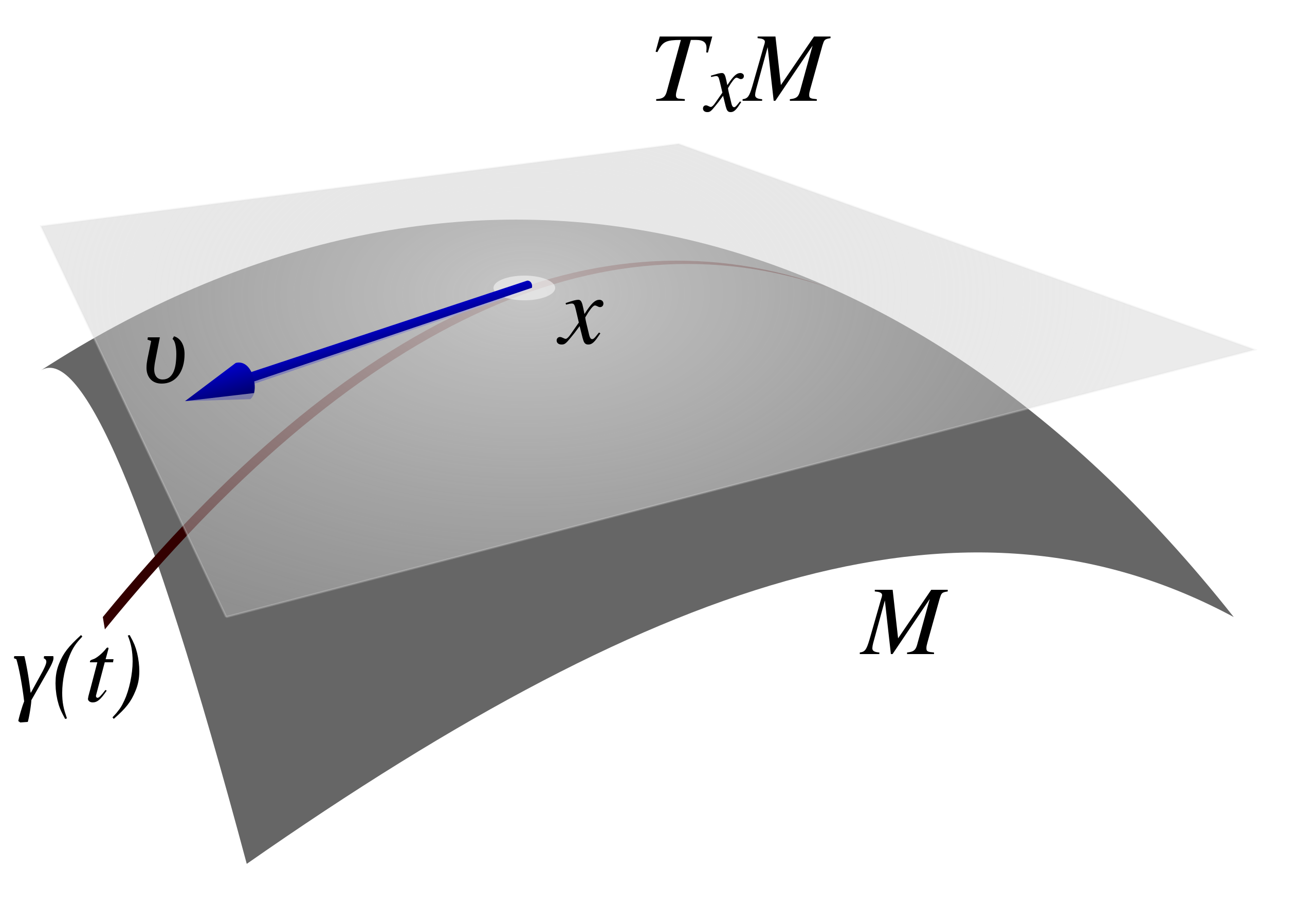}
\caption{In this figure, $M$ is a Riemannian manifold in $\mathbb{R}^3$, $x\in M$ is an element of the manifold, and  $T_xM$ is the tangent space of $M$ at $x$ (the set of all tangent vectors to all smooth curves on $M$ passing through $x$). $\gamma(t)$ is a $\mathcal{C}^1$ curve passing through $x$, and $v\equiv v(x)=(v^i(x))_{i\in\{1,2,3\}}$ is a vector field.}
\label{fig-riem}
\end{figure}


\begin{proof}[Proof of Theorem \ref{thm-main-1}]
    We begin by pointing out that, as a consequence of Lemma \ref{uniqueness}, $\underline{\mathbf{P}} $ and $\overline{\mathbf{P}} $ are (precise) Markov semigroups, and so the infinitesimal generators $\underline{L} $ and $\overline{L} $ are well-defined.

    We first focus on the greatest element $\overline{\mathbf{P}} $ of $\mathbf{Q}$. Recall that the variance $\text{Var}_{\overline{\mu} }(f)$ of function $f\in\mathbb{L}^2(\overline{\mu})$ 
    is given by 
    \begin{equation}\label{variance}
        \text{Var}_{\overline{\mu} }(f)\coloneqq \int_E f^2 \text{d}\overline{\mu}  - \left( \int_E f  \text{d}\overline{\mu}  \right)^2,
    \end{equation}
    and that the Dirichlet form $\overline{\mathcal{E}} (f)$ at $f\in\mathcal{A}_0$ of the carré du champ operator $\overline{\Gamma} $ of the Markov generator $\overline{L} $ is given by 
    $$\overline{\mathcal{E}} (f) \coloneqq \int_E \overline{\Gamma} (f) \text{d} \overline{\mu}.$$
    
    Pick any $f\in\mathcal{A}_0$. Given our hypotheses, by \citet[Theorem 3.2.3]{bakry}, we have that for every $t\geq 0$, the following gradient bound holds
    \begin{equation}\label{eq-grad-bd}
        \overline{\Gamma} \left(\overline{P} _t f \right) \leq e^{-2t \rho_{\overline{L}}} \overline{P} _t\left( \overline{\Gamma} (f) \right).
    \end{equation}
    Then, by \citet[Theorem 4.7.2]{bakry}, equation \eqref{eq-grad-bd} implies that the following local Poincaré inequality holds for all $t\geq 0$,
    \begin{equation}\label{local-poinc-ineq}
        \overline{P} _t(f^2) - \left( \overline{P} _t f \right)^2 \leq \frac{1-e^{-2t \rho_{\overline{L}}}}{\rho_{\overline{L}}} \overline{P} _t\left( \overline{\Gamma} (f) \right).
    \end{equation}
    Here the factor $\frac{1-e^{-2t \rho_{\overline{L}}}}{\rho_{\overline{L}}}$ needs to be understood as $2t$ when $\rho_{\overline{L}}=0$.
    Following \citet[Theorem 2.35]{handel}, 
    taking the limit as $t\rightarrow \infty$ on both sides of  \eqref{local-poinc-ineq} yields the following Poincaré inequality
    \begin{equation}\label{poinc-ineq}
        \text{Var}_{\overline{\mu} }(f) \leq \frac{1}{\rho_{\overline{L}}} \overline{\mathcal{E}} (f). 
    \end{equation}
    Here it is implicitly assumed that $\rho_{\overline{L}}>0$; We expand on this in Remark \ref{techn-rem}.
    In addition, recall that for the Dirichlet form $\overline{\mathcal{E}} (f)$
    we have that
    \begin{equation}\label{int-by-parts}
        \overline{\mathcal{E}} (f)=-\int_E f \overline{L}  f \text{d} \overline{\mu} .
    \end{equation}
    Combining \eqref{int-by-parts} with \eqref{poinc-ineq}, we have that if function $f$ is such that $\overline{L}  f =0$, then $\overline{\mathcal{E}} (f)=0$. In turn, this implies that $\text{Var}_{\overline{\mu} }(f)=0$, and so the function $f$ must be constant $\overline{\mu} $-a.e. 
 Moreover, $\mathcal A_0$ is a core for $\overline L$, i.e. it is dense in $\mathcal D(\overline L)$ for the graph norm
$\|f\|_{\mathcal D(\overline L)}=\|f\|_{\mathbb L^2(\overline\mu)}+\|\overline L f\|_{\mathbb L^2(\overline\mu)}$.\footnote{This norm and its properties are recalled in Appendix \ref{diff-mak-tr-back}.}
Therefore, equations \eqref{variance}-\eqref{int-by-parts} established on $\mathcal A_0$ extend to $\mathcal D(\overline L)$ by approximation.
This implies that every $f\in\mathcal{D}(\overline{L} )$ such that $\overline{L} f=0$ is constant, and so that $\overline{L} $ is ergodic according to \citet[Definition 3.1.11]{bakry}. Then, since $\overline{\mu} $ is finite,
    by \citet[Proposition 3.1.13]{bakry} we have that
    \begin{equation*}
        \forall f \in \mathbb{L}^2(\overline{\mu} ), \quad \lim_{t \rightarrow \infty} \overline{P}_t  f =\int_E f \text{d}\overline{\mu} , \quad \overline{\mu} \text{-a.e.}
    \end{equation*}
    in $\mathbb{L}^2(\overline{\mu} )$. This concludes the first part of the proof. The second part, that is, the one pertaining the least element $\underline{\mathbf{P}} $ of $\mathbf{Q}$, is analogous. The last statement of the theorem is immediate once we let $\overline{\mu} =\underline{\mu} =\mu $.
\end{proof}

\begin{remark}\label{techn-rem}
    Let us add a technical remark to Theorem \ref{thm-main-1}. If we do not know whether the (reversible invariant) measures $\overline{\mu} $ and $\underline{\mu} $ are finite, 
then in the statement of condition (ii) of Theorem \ref{thm-main-1} we must require that $\rho_{\underline{L}},\rho_{\overline{L}}>0$ \citep[Theorem 3.2.7]{bakry}, so that \eqref{eq-grad-bd} holds.
\end{remark}

Theorem~\ref{thm-main-1} is about the chosen representatives $\underline{\mathbf{P}} \in \underline{[\mathbf{P}]}_{\tilde{f}}$ and $\overline{\mathbf{P}} \in\overline{[\mathbf{P}]}_{\tilde{f}}$ : the invariant measures $\underline{\mu}, \overline{\mu}$ and generators $\underline{L}, \overline{L}$ are associated with these representatives. The proof uses only assumptions (i)-(iii) on the representative generator; Extremality is not used beyond selecting which semigroups are considered. Hence, the same convergence conclusion holds for any representative in $\mathbf{Q}$ (in particular, in the extremal classes) satisfying (i)-(iii). Moreover, if two representatives $\mathbf{P}, \mathbf{P}^{\prime}$ lie in the same class $[\mathbf{P}]=\left[\mathbf{P}^{\prime}\right]$ and both satisfy the conclusion with invariant measures $\mu, \mu^{\prime}$, then they necessarily satisfy $\int \tilde{f} d \mu=\int \tilde{f} d \mu^{\prime}$, since $P_t \tilde{f}=P_t^{\prime} \tilde{f}$ for all $t$. 

Examples of semigroups for which the conditions (i)-(iii) of Theorem \ref{thm-main-1} are met are the Ornstein-Uhlenbeck and Laguerre semigroups (whose Markov generators satisfy the Bakry-Émery curvature condition with $\rho>0$; In particular, $\rho=1$ for the former, and $\rho=1/2$ for the latter). 
In addition, an example of two Markov semigroups having the same invariant measure (as required for \eqref{ergodic-equiv-imp} to hold) are the Brownian motion with quadratic potential, and the Ornstein-Uhlenbeck semigroup with volatility parameter equals to $\sqrt{2}$. Indeed, for both, the invariant measure is a Gaussian with mean $0$, and variance the inverse of the rate of mean reversion.

The following is an important corollary that subsumes the ergodic behavior of the elements of $\mathbf{Q}$ when applied to $\tilde{f}\in B(E)$. We denote by $\underline\mu$, $\overline\mu$ the invariant measures of the chosen representatives of the extremal classes from Theorem~\ref{thm-main-1}.
\begin{corollary}[Ergodicity of Compact Imprecise Markov Semigroup $\mathbf{Q}$, $E$  Euclidean space or Riemannian manifold]\label{limiting-cor}
    Retain the assumptions of Theorem \ref{thm-main-1}, and let $\tilde{f}\in B(E)$ be the same functional that we fixed in Theorem \ref{thm-main-1}. Then, for all $\mathbf{P}=(P_t)_{t\geq 0} \in \mathbf{Q}$, we have that
    \begin{align}\label{cor-1-res}
        \lim_{t\rightarrow\infty} P_t \tilde{f} \geq \int_E \tilde{f}\text{d}\underline{\mu}, \quad \underline{\mu}\text{-a.e.}
    \end{align}
    in $\mathbb{L}^2(\underline{\mu})$, and
    \begin{align}\label{cor-2-res}
        \lim_{t\rightarrow\infty} P_t \tilde{f} \leq \int_E \tilde{f}\text{d}\overline{\mu}, \quad \overline{\mu}\text{-a.e.}
    \end{align}
    in $\mathbb{L}^2(\overline{\mu})$. In addition, if the invariant measure for $\underline{\mathbf{P}} $ is equal to that of $\overline{\mathbf{P}} $, and if we denote it by $\mu $, then
    \begin{align}\label{cor-3-res}
        \lim_{t\rightarrow\infty} P_t \tilde{f} = \int_E \tilde{f}\text{d}{\mu}, \quad {\mu}\text{-a.e.}
    \end{align}
    in $\mathbb{L}^2(\mu )$.
\end{corollary}

\begin{proof}
    Immediate from Theorem \ref{thm-main-1}, and the fact that -- as a consequence of Lemma \ref{uniqueness} and Definition \ref{cpct-imsg-def} -- $\overline{\mathbf{P}}$ and $\underline{\mathbf{P}}$ are the greatest and least elements of $\mathbf{Q}$, respectively.
\end{proof}


Let us add a few comments on Corollary~\ref{limiting-cor}. 
If $\mu$ is an invariant measure for a (precise) semigroup $\mathbf P=(P_t)_{t\ge0}$, then by definition $\mu P_t=\mu$ for all $t\ge0$; Equivalently,
$\int_E P_t f \mathrm d\mu=\int_E f \mathrm d\mu$ for every bounded $f$. 
This is a \emph{stationarity} statement and holds for every $t$ as soon as the initial law is chosen to be $\mu$.
Corollary~\ref{limiting-cor}, however, is of a different nature: It provides \emph{long-term} bounds for the conditional expectation
$P_t\tilde f(x)=\mathbb E[\tilde f(X_t)\mid X_0=x]$ (uniformly in the initial state, and robustly over $\mathbf Q$), for a fixed $\tilde f$, which is not implied by invariance alone and typically requires mixing/ergodicity assumptions.
In particular, in the case $\underline{\mu}=\overline{\mu}=\mu$, the bounds collapse to the same constant,
\[
\lim_{t\to\infty}\underline P_t\tilde f(x)=\lim_{t\to\infty}\overline P_t\tilde f(x)=\int_E \tilde f \mathrm d\mu,
\]
so the imprecision about the long-run expectation of $\tilde f$ vanishes, even though the short-run values may differ across models in $\mathbf Q$.
Finally, we emphasize that the limits in Corollary~\ref{limiting-cor} are \emph{not} Ces\`aro time averages along a single trajectory; Rather, they concern convergence of the semigroup (one-time marginals) as $t\to\infty$.
Motivated by \eqref{cor-1-res}--\eqref{cor-3-res}, we say that $\mathbf Q$ is $\tilde f$-\emph{lower ergodic} if it satisfies \eqref{cor-1-res},
$\tilde f$-\emph{upper ergodic} if it satisfies \eqref{cor-2-res}, and $\tilde f$-\emph{ergodic} if it satisfies \eqref{cor-3-res}.

As we pointed out in Section \ref{prelim}, there are two main differences with respect to the existing literatures on imprecise Markov processes and ergodicity for imprecise probabilities. We are the first to explicitly exploit the geometry of the state space $E$ via the Bakry-Émery curvature condition to derive our results. 
In addition, in Theorem \ref{thm-main-1} and Corollary \ref{limiting-cor} we represent the scholar's ambiguity about the invariant measure and the transition probability via a compact imprecise Markov semigroup $\mathbf{Q}$.
We briefly mention that our approach differs slightly from the imprecise ergodic literature \citep{KozineUtkin2002intervalValuedFiniteMarkovChains,HermansDeCooman2012ergodicUpperTransitionOperators,Skulj2013classificationInvariantDistributions,Skulj2015efficientComputationCTIMC,cerreia,DECOOMAN201618,dipk,ergo_me}, where ambiguity around measure $\mu$ is typically represented via a Choquet approach involving a superadditive Choquet capacity.

We conclude this section with a corollary that further informs us on the geometrical properties of $E$ in the ergodic regime of $\mathbf{Q}$, when $E$ is a complete Riemannian manifold. Let us denote by $E^n_K$ the complete $n$-dimensional simply connected space of constant sectional curvature $K \in \mathbb{R}$,\footnote{The Ricci curvature 
 of $E^n_K$ is constant and equal to $(n-1)K$.}
and by $B(p,r)$ the ball of radius $r$ at point $p$, defined with respect to the Riemannian distance function. Once again, the values that we study are tied to the choice of the representatives $\underline{\mathbf{P}}$ and $\overline{\mathbf{P}}$.

\begin{corollary}[Bishop-Gromov Inequality]\label{bishop-gromov}
    Let $E \subseteq \mathbb{R}^n$ be a complete Riemannian manifold. Retain the first part of Theorem \ref{thm-main-1}, and assume that condition (ii) holds. Then,
    $$\text{Vol}\left[ B(p,r) \right] \leq \text{Vol}\left[ B(p_K,r) \right],$$
    for all $p\in E$, all $p_K \in E^n_K$, all $r \in (0,\infty)$, and some $K$ that depends on $\rho_{\underline{L}}$ and $\rho_{\overline{L}}$. Here, $\text{Vol}[\cdot]$ denotes the volume operator.
\end{corollary}

\begin{proof}
    Given our assumptions, we know that $CD(\rho_{\underline{L}},\infty)$ and $CD(\rho_{\overline{L}},\infty)$ both hold. Let us define $\rho_\star \coloneqq \min\{\rho_{\underline{L}}, \rho_{\overline{L}}\}$, and put $\rho_\star = (n-1)K$, so that $K= \rho_\star/(n-1)$. The Bakry-Émery curvature assumptions imply that the Ricci tensor at every point of $E$ is bounded from below by $\rho_\star$.\footnote{This is discussed in Appendix \ref{riemann-manif}.} In turn, we have that the Ricci curvature $\text{Ric}$ of $E$ is lower bounded by $(n-1)K$, that is, $\text{Ric} \geq (n-1)K$. The result then follows from \citet{bishop}.
    \end{proof} 

Corollary \ref{bishop-gromov} tells us that -- under the Bakry-Émery curvature condition (ii) of Theorem \ref{thm-main-1} -- the volume of a ball around any element of the manifold $E$ is upper bounded by the volume of a ball around any element of $E^n_K$ having the same radius. Further comparison results like Corollary \ref{bishop-gromov} that hold in the ergodic regime of $\mathbf{Q}$ will be the subject of future study.

\section{Ergodicity in the General Case -- Compact $\mathbf{Q}$}\label{idmt-section}
In this section, our goal is to extend Theorem \ref{thm-main-1} and Corollary \ref{limiting-cor} to the case where $(E,\mathcal{F})$ is an arbitrary good measurable space, while retaining the assumptions that $\mathbf{Q}$ is compact and $\tilde f$ is fixed. To do so, we need to introduce the concept of an Imprecise Diffusion Markov Tuple, the imprecise version of a diffusion Markov triple; The latter is recalled in Appendix \ref{diff-mak-tr-back}. Similarly to Section \ref{riemann}, in the following definition, theorem, and corollary we have that 
$\underline{\mathbf{P}} \equiv \underline{\mathbf{P}}_{\tilde{f}}$, a fixed chosen representative from $\underline{[\mathbf{P}]}_{\tilde{f}}$, and similarly for $\overline{\mathbf{P}}$.

\begin{definition}[Imprecise Diffusion Markov (5-)Tuple]\label{idmt-def}
    Fix any $\tilde{f}\in B(E)$, and assume that we can define the partial order $\preceq_{\tilde{f}}^\text{part}$ or the total order $\preceq_{\tilde{f}}^\text{tot}$ that we introduced in Section \ref{prelim}. An \textit{imprecise diffusion Markov tuple} $(E,\overline{\mu} ,\underline{\mu} ,\overline{\Gamma} ,\underline{\Gamma} )$ is a 5-tuple such that the following holds.
    \begin{itemize}
        \item $(E,\mathcal{F})$ is a good measurable space.
        \item $\overline{\mu} $ and $\underline{\mu} $ are ($\sigma$-finite) measures on $\mathcal{F}$.
        \item $\mathcal{A}^{\overline{\mu}}_0$ and $\mathcal{A}^{\underline{\mu}}_0$ are vector spaces of bounded measurable functionals dense in all $\mathbb{L}^p(\overline{\mu} )$ and all $\mathbb{L}^p(\underline{\mu} )$ spaces, $1\leq p < \infty$, respectively, and stable under products. 
        \item $(E,\overline{\mu} ,\overline{\Gamma} )$ and  $(E,\underline{\mu} ,\underline{\Gamma} )$ are diffusion Markov triples having associated algebras $\mathcal{A}^{\overline{\mu}}_0$ and $\mathcal{A}^{\underline{\mu}}_0$, respectively, and carré du champs operators $\overline{\Gamma} $ and $\underline{\Gamma} $, respectively.
        \item The Markov semigroups $\overline{\mathbf{P}} =(\overline{P} _t)_{t \geq 0}$ and $\underline{\mathbf{P}} =(\underline{P} _t)_{t \geq 0}$ associated with $\overline{\Gamma} $ and $\underline{\Gamma} $ are the greatest and least elements of a compact imprecise Markov semigroup $\mathbf{Q}$,\footnote{With respect to $\preceq^\text{part}_{\tilde{f}}$ or $\preceq^\text{tot}_{\tilde{f}}$.} respectively.
    \end{itemize}
\end{definition}

We are now ready for the main result of this section. It hinges on the concepts of adjoint operator, self-adjointness, essential self-adjointness (ESA), extended algebra $\mathcal{A}$ of the algebra $\mathcal{A}_0$ associated with a diffusion Markov triple, connexity, and weak hypo-ellipticity. They are brushed off in Appendix \ref{diff-mak-tr-back}.
As its proof technique 
is similar to that of Theorem \ref{thm-main-1} (despite exploiting different, more general results), we defer it to Appendix \ref{app-a}.

\begin{theorem}[Limiting Behavior of $\underline{\mathbf{P}}$ and $\overline{\mathbf{P}}$, General]\label{thm-main-2}
    Fix any $\tilde{f}\in B(E)$, and assume that we can define the partial order $\preceq_{\tilde{f}}^\text{part}$ or the total order $\preceq_{\tilde{f}}^\text{tot}$ that we introduced in Section \ref{prelim}. Let $(E,\overline{\mu} ,\underline{\mu} ,\overline{\Gamma} ,\underline{\Gamma} )$ be an imprecise diffusion Markov tuple, and denote by $\mathcal{A}^{\overline{\mu}}$ and $\mathcal{A}^{\underline{\mu}}$ the extended algebras of $\mathcal{A}_0^{\overline{\mu}}$ and $\mathcal{A}_0^{\underline{\mu}}$ associated with $(E,\overline{\mu} ,\overline{\Gamma} )$ and  $(E,\underline{\mu} ,\underline{\Gamma} )$, respectively. If the connexity and weak hypo-ellipticity assumptions hold for both $(E,\overline{\mu} ,\overline{\Gamma} )$ and  $(E,\underline{\mu} ,\underline{\Gamma} )$, and if the Bakry-Émery curvature conditions $CD(\rho_{\overline{L}},\infty)$ and $CD(\rho_{\underline{L}},\infty)$ are met by both the infinitesimal generators $\overline{L} $ and $\underline{L} $ associated with $\overline{\Gamma} $ and $\underline{\Gamma} $, respectively, for some $\rho_{\overline{L}},\rho_{\underline{L}}\in\mathbb R$, then the following hold
    $$\forall f \in \mathbb{L}^2(\underline{\mu} ), \quad \lim_{t \rightarrow \infty} \underline{P}_t  f = \mathbb{E}_{\underline{\mu} }[f] = \int_E f \text{d}\underline{\mu}  \quad \text{in } \mathbb{L}^2(\underline{\mu} ), \quad \underline{\mu} \text{-a.e.}$$
    and
    $$\forall f \in \mathbb{L}^2(\overline{\mu} ), \quad \lim_{t \rightarrow \infty} \overline{P}_t  f = \mathbb{E}_{\overline{\mu} }[f] = \int_E f \text{d}\overline{\mu}  \quad \text{in } \mathbb{L}^2(\overline{\mu} ), \quad \overline{\mu} \text{-a.e.}$$
    In addition, if $\overline{\mu} =\underline{\mu} =\mu $, then
    $$\forall f \in \mathbb{L}^2(\mu ), \quad \lim_{t \rightarrow \infty} \overline{P}_t  f =  \lim_{t \rightarrow \infty} \underline{P}_t  f =\mathbb{E}_{\mu }[f] = \int_E f \text{d}\mu  \quad \text{in } \mathbb{L}^2(\mu ), \quad \mu \text{-a.e.}$$
\end{theorem}

The connexity and weak hypo-ellipticity assumptions allow to ``import'' some of the geometric structure of the Riemannian manifolds considered in Section \ref{riemann} to the more general diffusion Markov triple setting, and to relax the requirement in Theorem \ref{thm-main-1} that $\underline{L} $ and $\overline{L} $ are elliptic operators. This is expanded upon in Appendix \ref{diff-mak-tr-back}.

Similarly to what we pointed out in Remark \ref{techn-rem}, if we do not know whether the (reversible invariant) measures $\overline{\mu} $ and $\underline{\mu} $ are finite, then in the statement of Theorem \ref{thm-main-2} we must require that $\rho_{\overline{L}}, \rho_{\underline{L}}>0$ \citep[Theorem 3.3.23]{bakry}.

Similarly to Corollary \ref{limiting-cor}, the following is an important result that subsumes the ergodic behavior of the elements of a compact IMSG $\mathbf{Q}$ when applied to $\tilde{f}\in B(E)$.

\begin{corollary}[Ergodicity of Compact Imprecise Markov Semigroup $\mathbf{Q}$, General]\label{limiting-cor2}
    Retain the assumptions of Theorem \ref{thm-main-2}, and let $\tilde{f}\in B(E)$ be the same functional that we fixed in Theorem \ref{thm-main-2}. Then, for all $\mathbf{P}=(P_t)_{t\geq 0} \in \mathbf{Q}$, we have that
    $$\lim_{t\rightarrow\infty} P_t \tilde{f} \geq \int_E \tilde{f}\text{d}\underline{\mu}, \quad \underline{\mu}\text{-a.e.}$$
    in $\mathbb{L}^2(\underline{\mu})$, and
    $$\lim_{t\rightarrow\infty} P_t \tilde{f} \leq \int_E \tilde{f}\text{d}\overline{\mu}, \quad \overline{\mu}\text{-a.e.}$$
    in $\mathbb{L}^2(\overline{\mu})$. In addition, if the invariant measure for $\underline{\mathbf{P}} $ is equal to that of $\overline{\mathbf{P}} $, and if we denote it by $\mu $, then
    $$\lim_{t\rightarrow\infty} P_t \tilde{f} = \int_E \tilde{f}\text{d}{\mu}, \quad {\mu}\text{-a.e.}$$
    in $\mathbb{L}^2(\mu )$.
\end{corollary}

\begin{proof}
    Immediate from Theorem \ref{thm-main-2}, and the fact that -- as a consequence of Lemma \ref{uniqueness} and Definition \ref{cpct-imsg-def} -- $\overline{\mathbf{P}}$ and $\underline{\mathbf{P}}$ are the greatest and least elements of the compact IMSG $\mathbf{Q}$, respectively.
\end{proof}

\section{Ergodicity in the General Case -- $\mathbf{Q}$ Not Necessarily Compact}\label{subsec:wass-sublinear}

In this section, we show that the long-term conclusions suggested by our imprecise Markov semigroup
viewpoint can be obtained \emph{without} requiring the existence of least/greatest semigroups in
$\mathbf Q$ (hence without any compactness assumption; We also do without the notational convention of Remark \ref{quotient-notation}). To do so, we need the following.


Let $(E,d)$ be a Polish metric space, $\mathcal F$ its Borel $\sigma$-algebra, and assume
\begin{equation}\label{eq:diam-ass}
\mathrm{diam}(E)\coloneqq \sup_{x,y\in E} d(x,y) \eqqcolon D <\infty.
\end{equation}
For instance, \eqref{eq:diam-ass} holds when $E$ is compact under $d$. 
To avoid measurability pathologies for pointwise suprema over uncountable index sets (and over switching times), we work with
the universal completion $\mathcal F^\star$ of $\mathcal{F}$, and interpret envelopes as universally measurable; 
We expand on this in Appendix \ref{rem:meas}. We require the following.

\begin{assumption}[Standard Borel Parametrization and Jointly Measurable Kernels]\label{ass:SB-wass}
The family $\mathbf Q$ satisfying \eqref{imsg-def-1} is parameterized as $\mathbf Q=\{\mathbf P^\theta\}_{\theta\in\Theta}$ for a standard
Borel space $\Theta$, where $\mathbf P^\theta=(P_t^\theta)_{t\ge0}$ is a Markov semigroup with transition
kernels $p_t^\theta(x,\mathrm dy)$. For every $A\in\mathcal F$, the map $(t,x,\theta)\mapsto p_t^\theta(x,A)$
is Borel measurable on $\mathbb{R}_+\times E\times\Theta$.
\end{assumption}

Assumption~\ref{ass:SB-wass} is a standard measurability hypothesis from stochastic control:
It rules out pathological nonmeasurable envelopes when taking the supremum over $\theta$ and when switching along partitions.
It is typically satisfied for parametrized SDE/semigroup families with measurable dependence on $\theta$, and imposes
no compactness requirements on $\Theta$.

Let $\Delta_E$ be the set of Borel probability measures on $(E,\mathcal F)$ and
$$\Delta_E^1\coloneqq\{\mu\in\Delta_E: \int_E d(x_0,x)\mu(\mathrm dx)<\infty\}$$ 
for some (hence any) $x_0\in E$.
Under \eqref{eq:diam-ass}, we have $\Delta_E^1=\Delta_E$.
For a bounded function $f:E\to\mathbb R$ define the oscillation
\[
\osc(f)\coloneqq \sup_{x\in E} f(x)-\inf_{x\in E} f(x).
\]
Denote by $\mathrm{Lip}_b(E)$ the space of bounded $d$-Lipschitz functions, equipped with the seminorm
\[
\mathrm{Lip}(f)\coloneqq \sup_{x\neq y}\frac{|f(x)-f(y)|}{d(x,y)}\in[0,\infty).
\]

Let us now recall the definition of the 1-Wasserstein metric; It is instrumental to the second assumption we need to study the ergodicity of a non-compact $\mathbf{Q}$.

\begin{definition}[1-Wasserstein Distance]\label{1-wass-def}
    For $\mu,\nu\in\Delta_E^1$, let $W_1(\mu,\nu)$ denote the 1-Wasserstein distance, characterized by the
Kantorovich-Rubinstein duality
\begin{equation}\label{eq:KR}
W_1(\mu,\nu)=\sup\left\lbrace{\int_E \varphi \mathrm d(\mu-\nu): \varphi\in\mathrm{Lip}_b(E), \mathrm{Lip}(\varphi)\le 1}\right\rbrace.
\end{equation}

\noindent Equivalently,
\[
W_1(\mu,\nu)=\sup\left\lbrace{\left|\int_E \varphi  \mathrm d(\mu-\nu)\right|: \varphi\in\mathrm{Lip}_b(E), \mathrm{Lip}(\varphi)\le 1}\right\rbrace,
\]
since $\varphi$ is admissible if and only if $-\varphi$ is admissible, that is, if $\varphi$ has Lipschitz constant $\leq 1$, then $-\varphi$ also has Lipschitz constant $\leq 1$, hence both are allowed as test functions in the supremum.
\end{definition}

\begin{assumption}[Uniform Exponential 1-Wasserstein Contraction]\label{ass:UEWC}
There exists $\lambda>0$ such that for every $\theta\in\Theta$ and every $t\ge0$,
\begin{equation}\label{eq:UEWC}
W_1(\mu P_t^\theta,\nu P_t^\theta)\le e^{-\lambda t} W_1(\mu,\nu),
\qquad \forall \mu,\nu\in\Delta_E^1,
\end{equation}
where $\mu P_t^\theta(A) \coloneqq \int_E p_t^\theta(x,A) \mu(\text{d}x)$, for all $A \in\mathcal{F}$, and similarly for $\nu P_t^\theta$.
\end{assumption}



Assumption~\ref{ass:UEWC} is a \emph{uniform exponential mixing/stability condition in the $1$-Wasserstein distance $W_1$}:
For every candidate model $\mathbf P^\theta$, the semigroup contracts distances between laws at rate
$e^{-\lambda t}$, uniformly over $\theta$.

In many diffusion settings, a Bakry-\'Emery curvature lower bound can yield a contraction of the form
\eqref{eq:UEWC} (possibly in $W_1$, in the $2$-Wasserstein distance $W_2$, or for the intrinsic distance associated with the carré du champ),
via gradient estimates and the duality between gradient bounds and Wasserstein contraction. The following is an example.


\begin{example}[Ornstein-Uhlenbeck Family: A $\theta$-Parametrized Curvature-Driven $W_1$ Contraction]\label{ex:OU-w1}
Let $E=\mathbb{R}^d$  with Euclidean distance $d(x,y)=\|x-y\|$. Of course, $\mathbb{R}^d$ does not have finite diameter: Assumption \eqref{eq:diam-ass} is not needed for this example.
Fix a parameter set $\Theta$ and, for each $\theta\in\Theta$, consider the Ornstein-Uhlenbeck generator
\[
L^\theta f(x)=\mathrm{tr} \left(a^\theta \nabla^2 f(x) \right)-\langle b^\theta x,\nabla f(x)\rangle,
\]
where $a^\theta$ is a symmetric positive semidefinite $d\times d$ matrix and $b^\theta$ is a $d\times d$ matrix
whose spectrum is contained in $(0,\infty)$ (so the drift is mean-reverting).
Let $\mathbf P^\theta=(P_t^\theta)_{t\ge0}$ be the associated Markov semigroup.

Equivalently, $\mathbf P^\theta$ is the transition semigroup of the linear SDE
\[
\mathrm dX_t=-b^\theta X_t \mathrm dt+\sqrt{2a^\theta} \mathrm dW_t,
\]
where $(W_t)_{t\ge0}$ is a standard $d$-dimensional Brownian motion, and the Mehler formula yields, for bounded measurable $f$,
\[
P_t^\theta f(x)=\mathbb E\Big[f\big(e^{-t b^\theta}x+Z_t^\theta\big)\Big],
\qquad
Z_t^\theta\sim \mathcal N \big(0,\Sigma_t^\theta\big),
\]
where $\mathcal N \big(0,\Sigma_t^\theta\big)$ denotes a Gaussian distribution with mean $0$ and covariance 
$$\Sigma_t^\theta=\int_0^t e^{-s b^\theta}(2a^\theta)e^{-s(b^\theta)^\top} \mathrm ds.$$

If $f\in \mathrm{Lip}_b(E)$, then for all $x,y\in\mathbb{R}^d$,
\[
|P_t^\theta f(x)-P_t^\theta f(y)|
\le \mathbb E\Big[\big|f(e^{-t b^\theta}x+Z_t^\theta)-f(e^{-t b^\theta}y+Z_t^\theta)\big|\Big]
\le \mathrm{Lip}(f) \|e^{-t b^\theta}(x-y)\|.
\]
Hence
\[
\mathrm{Lip}(P_t^\theta f)\le \|e^{-t b^\theta}\|_{\mathrm{op}} \mathrm{Lip}(f)\qquad\forall t\ge0,
\]
where $\|\cdot\|_{\mathrm{op}}$ denotes the operator norm. If, in addition, there exists $\lambda>0$ such that
\[
\sup_{\theta\in\Theta}\|e^{-t b^\theta}\|_{\mathrm{op}}\le e^{-\lambda t}\qquad\forall t\ge0,
\]
(e.g. it suffices that $\inf_{\theta\in\Theta}\lambda_{\min} \big((b^\theta+(b^\theta)^\top)/2\big)\ge \lambda$),
then for all $\theta\in\Theta$ and $t\ge0$,
\[
\mathrm{Lip}(P_t^\theta f)\le e^{-\lambda t}\mathrm{Lip}(f).
\]
By Kantorovich-Rubinstein duality, this is equivalent to the uniform $1$-Wasserstein contraction
\[
W_1(\mu P_t^\theta,\nu P_t^\theta)\le e^{-\lambda t}W_1(\mu,\nu),
\qquad \forall \mu,\nu\in\Delta_E^1, \forall t\ge0, \forall \theta\in\Theta,
\]
i.e. Assumption~\ref{ass:UEWC} holds with such a  $\lambda$.
\end{example}


Example \ref{ex:OU-w1} can be viewed as a
prototype of the general principle ``Bakry-\'Emery curvature lower bounds $\implies$ gradient estimates for $(P_t)_{t\ge0}$
$\implies$ Wasserstein contraction of  $\mu\mapsto \mu P_t$'', with a precise equivalence between suitable gradient bounds and $W_p$-contraction given by
Kuwada-type duality \citep{kuwada2010,savare2014}. We formalize such a principle in Lemma \ref{lem:BE-implies-UEWC}.

In this section, we take \eqref{eq:UEWC} as a compactness-free replacement for the existence of extremal semigroups:
It is strong enough to yield uniform ergodic control (hence a robust limit) for the sublinear Nisio semigroup introduced in Section \ref{nisio-section}.
Moreover, \eqref{eq:UEWC} is verifiable in concrete model classes via metric contractive couplings or,
in diffusion settings, via gradient estimates (e.g.  stemming from Bakry-\'Emery curvature lower bounds) and
their dual formulation in Wasserstein distance.
We now record the dual (Lipschitz) form of Assumption~\ref{ass:UEWC}.

\begin{lemma}[Wasserstein Contraction $\iff$ Lipschitz Contraction]\label{lem:wass-lip-equiv}
We have that \eqref{eq:UEWC} holds if and only if
\begin{equation}\label{eq:lip-contract}
\mathrm{Lip}(P_t^\theta f)\le e^{-\lambda t} \mathrm{Lip}(f),
\qquad \forall f\in\mathrm{Lip}_b(E), \forall \theta\in\Theta, \forall t\ge0.
\end{equation}
\end{lemma}

\begin{proof}
Assume \eqref{eq:UEWC}. Fix $\theta,t$ and $f\in\mathrm{Lip}_b(E)$.
For $x\neq y$, by \eqref{eq:KR} and \eqref{eq:UEWC},
\[
|P_t^\theta f(x)-P_t^\theta f(y)|
=\left|\int_E f \mathrm d(\delta_x P_t^\theta-\delta_y P_t^\theta)\right|
\le \mathrm{Lip}(f) W_1(\delta_x P_t^\theta,\delta_y P_t^\theta)
\le \mathrm{Lip}(f) e^{-\lambda t} d(x,y).
\]
Taking the supremum over $x\neq y$ gives \eqref{eq:lip-contract}.

Conversely, assume \eqref{eq:lip-contract}. Fix $\mu,\nu\in\Delta_E^1$ and $\varphi\in\mathrm{Lip}_b(E)$ with
$\mathrm{Lip}(\varphi)\le 1$. Then, by \eqref{eq:lip-contract}, $\mathrm{Lip}(P_t^\theta \varphi)\le e^{-\lambda t}$ and
\[
\left|\int_E \varphi \mathrm d(\mu P_t^\theta-\nu P_t^\theta)\right|
=\left|\int_E (P_t^\theta\varphi) \mathrm d(\mu-\nu)\right|
\le \mathrm{Lip}(P_t^\theta\varphi) W_1(\mu,\nu)
\le e^{-\lambda t} W_1(\mu,\nu).
\]
Taking the supremum over $\varphi$ with $\mathrm{Lip}(\varphi)\le 1$ and using \eqref{eq:KR} yields \eqref{eq:UEWC}.
\end{proof}

The next result uses the equivalence in Lemma \ref{lem:wass-lip-equiv} to show that a (uniform) Bakry-Émery curvature condition on $\mathbf{Q}$ implies our Assumption \ref{ass:UEWC}.

\begin{lemma}[Bakry-\'Emery Curvature $\implies$ $W_1$-Contraction]\label{lem:BE-implies-UEWC}
Let $(E,\mu,\Gamma)$ be a diffusion Markov triple with Markov semigroup $(P_t)_{t\ge 0}$ and generator $L$.
Let $\mathcal A$ be an algebra of bounded functionals, stable under products and invariant under $(P_t)_{t\ge0}$,
such that $\Gamma$ is well-defined on $\mathcal A$.
Define the \emph{intrinsic (pseudo-)distance} on $E$ associated with $\Gamma$ by
\begin{equation}\label{eq:intrinsic-distance}
d_\Gamma(x,y)
\coloneqq
\sup\Bigl\{u(x)-u(y): u\in\mathcal A, \|\sqrt{\Gamma(u)}\|_\infty\le 1\Bigr\}\in[0,\infty].
\end{equation}
Assume that $d_\Gamma$ is finite and separates points, so that $(E,d_\Gamma)$ is a metric space, and that
$(E,d_\Gamma)$ is Polish with Borel $\sigma$-algebra $\mathcal F$. Assume also the following density condition,
\begin{align}\label{density-assumption}
\forall \varphi\in \mathrm{Lip}_b(E,d_\Gamma) \text{ with } \mathrm{Lip}_{d_\Gamma}(\varphi)\le 1,
\exists (f_n)_{n\ge 1}\subset \Bigl\{ f\in\mathcal A: \|\sqrt{\Gamma(f)}\|_\infty\le 1\Bigr\} \text{ such that } \|f_n-\varphi\|_\infty\to 0.
\end{align}

For $\nu_1,\nu_2$ Borel probability measures with finite first moment with respect to  $d_\Gamma$, let
$W_{1,d_\Gamma}(\nu_1,\nu_2)$ denote the $1$-Wasserstein distance associated with $d_\Gamma$
-- equivalently given by the Kantorovich-Rubinstein duality on $(E,d_\Gamma)$.

If the Bakry-\'Emery curvature condition $CD(\rho,\infty)$ holds for some $\rho\in\mathbb R$, then for all
$t\ge0$ and all such $\nu_1,\nu_2$,
\begin{equation}\label{eq:W1-contract-intrinsic}
W_{1,d_\Gamma}(\nu_1P_t,\nu_2P_t)\le e^{-\rho t} W_{1,d_\Gamma}(\nu_1,\nu_2).
\end{equation}

In particular, for an imprecise Markov semigroup
$\mathbf Q=\{\mathbf P^\theta=(P_t^\theta)_{t\ge0}\}_{\theta\in\Theta}$, assume that for each $\theta$ we have a
diffusion Markov triple $(E,\mu_\theta,\Gamma_\theta)$ with intrinsic distance $d_{\Gamma_\theta}$, and that
\[
d_{\Gamma_\theta}=d_\Gamma\quad\text{for all }\theta\in\Theta
\qquad\text{and}\qquad
\inf_{\theta\in\Theta}\rho_\theta\ge \lambda>0,
\]
where $CD(\rho_\theta,\infty)$ holds for the $\theta$-model. Then, Assumption~\ref{ass:UEWC} holds on $(E,d_\Gamma)$
with rate $\lambda$.
\end{lemma}

Assumption \eqref{density-assumption} is standard density property: It holds in many standard settings. For example, in the Riemannian manifold case it follows from smooth approximation of Lipschitz functions with almost-preserved Lipschitz constant \citep{AZAGRA20071370}.

\begin{proof}[Proof of Lemma \ref{lem:BE-implies-UEWC}]
\emph{Step 1: Gradient Estimate under $CD(\rho,\infty)$.}
Assume that the Bakry-\'Emery curvature condition $CD(\rho,\infty)$ holds.
Then the Bakry-\'Emery gradient estimate holds on $\mathcal A$: for all $t\ge 0$ and all $f\in\mathcal A$,
\begin{equation}\label{eq:BE-grad-est-intrinsic}
\Gamma(P_t f)\le e^{-2\rho t}  P_t\Gamma(f)
\qquad \mu\text{-a.e.}
\end{equation}

\emph{Step 2: Contraction of the $\Gamma$-Seminorm.}
If $\|\sqrt{\Gamma(f)}\|_\infty\le 1$, then $0\le \Gamma(f)\le 1$ $\mu$-a.e., hence by positivity and
constant-preservation of $P_t$,
\[
0\le P_t\Gamma(f)\le 1\quad \mu\text{-a.e.}
\]
Combining this with \eqref{eq:BE-grad-est-intrinsic} yields
\[
\Gamma(P_t f)\le e^{-2\rho t}\quad \mu\text{-a.e.},
\qquad\text{so}\qquad
\|\sqrt{\Gamma(P_t f)}\|_\infty\le e^{-\rho t}.
\]
Equivalently, if $\|\sqrt{\Gamma(f)}\|_\infty\le 1$, then
\begin{equation}\label{eq:Gamma-unitball-invariance}
\Bigl\|\sqrt{\Gamma\bigl(e^{\rho t}P_t f\bigr)}\Bigr\|_\infty\le 1.
\end{equation}

\emph{Step 3: Intrinsic Lipschitz Bound.}
We claim that for every $g\in\mathcal A$,
\begin{equation}\label{eq:Lip-intrinsic-bound}
\mathrm{Lip}_{d_\Gamma}(g)\le \|\sqrt{\Gamma(g)}\|_\infty.
\end{equation}
Indeed, if $\|\sqrt{\Gamma(g)}\|_\infty<\infty$, then for any $x,y\in E$ the function
$u\coloneqq g/\|\sqrt{\Gamma(g)}\|_\infty$ satisfies $\|\sqrt{\Gamma(u)}\|_\infty\le 1$, hence by the definition
\eqref{eq:intrinsic-distance} of $d_\Gamma$,
\[
\frac{g(x)-g(y)}{\|\sqrt{\Gamma(g)}\|_\infty}=u(x)-u(y)\le d_\Gamma(x,y).
\]
Swapping $x,y$ gives $|g(x)-g(y)|\le \|\sqrt{\Gamma(g)}\|_\infty  d_\Gamma(x,y)$, proving
\eqref{eq:Lip-intrinsic-bound}.

As a consequence of \eqref{eq:Lip-intrinsic-bound} and Step 2, for every $f\in\mathcal A$ with
$\|\sqrt{\Gamma(f)}\|_\infty<\infty$ we have the estimate
\begin{equation}\label{eq:Lip-contract-correct}
\mathrm{Lip}_{d_\Gamma}(P_t f)\le \|\sqrt{\Gamma(P_t f)}\|_\infty
\le e^{-\rho t} \|\sqrt{\Gamma(f)}\|_\infty.
\end{equation}

\emph{Step 4: $W_{1,d_\Gamma}$-Contraction via Kantorovich-Rubinstein Duality and a Core Density.}
Let
\[
\mathsf F \coloneqq \Bigl\{ f\in\mathcal A: \|\sqrt{\Gamma(f)}\|_\infty\le 1\Bigr\}.
\]
By \eqref{eq:Lip-intrinsic-bound}, every $f\in\mathsf F$ is $1$-Lipschitz w.r.t. $d_\Gamma$, hence is an
admissible test function in the Kantorovich-Rubinstein duality for $W_{1,d_\Gamma}$.

Fix $\nu_1,\nu_2$ with finite first moment with respect to $d_\Gamma$. For any $f\in\mathsf F$, define
$g_t \coloneqq e^{\rho t}P_t f$. By \eqref{eq:Gamma-unitball-invariance} we have $g_t\in\mathsf F$, and therefore
\[
\int_E P_t f  \mathrm d(\nu_1-\nu_2)
= e^{-\rho t}\int_E g_t  \mathrm d(\nu_1-\nu_2)
\le e^{-\rho t} W_{1,d_\Gamma}(\nu_1,\nu_2),
\]
where the last inequality follows from \eqref{eq:KR} since $\mathrm{Lip}_{d_\Gamma}(g_t)\le 1$.
Taking the supremum over $f\in\mathsf F$ yields
\begin{equation}\label{eq:sup-over-F}
\sup_{f\in\mathsf F}\int_E P_t f \mathrm d(\nu_1-\nu_2)
\le e^{-\rho t} W_{1,d_\Gamma}(\nu_1,\nu_2).
\end{equation}

To pass from the restricted supremum to the full Kantorovich-Rubinstein supremum, we use Assumption \eqref{density-assumption}. Let $\varphi\in \mathrm{Lip}_b(E,d_\Gamma)$ with $\mathrm{Lip}_{d_\Gamma}(\varphi)\le 1$ and pick $f_n\in\mathsf F$
with $\|f_n-\varphi\|_\infty\to 0$. Since $P_t$ is a contraction on $L^\infty$, we have that
$\|P_t f_n-P_t\varphi\|_\infty\le \|f_n-\varphi\|_\infty\to 0$. Hence, from \eqref{eq:sup-over-F} it follows that 
\[
\int_E P_t\varphi \mathrm d(\nu_1-\nu_2)
=\lim_{n\to\infty}\int_E P_t f_n \mathrm d(\nu_1-\nu_2)
\le e^{-\rho t} W_{1,d_\Gamma}(\nu_1,\nu_2).
\]
Taking the supremum over all such $\varphi$ and using
Kantorovich-Rubinstein duality \eqref{eq:KR} (with respect to $d_\Gamma$) gives
\[
W_{1,d_\Gamma}(\nu_1P_t,\nu_2P_t)\le e^{-\rho t} W_{1,d_\Gamma}(\nu_1,\nu_2),
\]
which is \eqref{eq:W1-contract-intrinsic}.

\emph{Final statement for $\mathbf Q$.}
For an imprecise Markov semigroup
$\mathbf Q=\{\mathbf P^\theta=(P_t^\theta)_{t\ge0}\}_{\theta\in\Theta}$, apply the above argument to each $\theta$.
If $d_{\Gamma_\theta}=d_\Gamma$ for all $\theta$ and $\inf_{\theta\in\Theta}\rho_\theta\ge \lambda>0$, then
\eqref{eq:W1-contract-intrinsic} holds with rate $\lambda$ uniformly in $\theta$, which is exactly
Assumption~\ref{ass:UEWC} on $(E,d_\Gamma)$.
\end{proof}

The key structural requirement in Lemma \ref{lem:BE-implies-UEWC} is the compatibility between the carré du champ and the metric, here encoded by the intrinsic distance $d_\Gamma$ -- or more generally by a domination of $\sqrt{\Gamma(f)}$ in terms of $\mathrm{Lip}(f)$. This is not overly restrictive: It is automatically satisfied in the classical Riemannian setting (where $\Gamma(f)=|\nabla f|^2$ and $d_\Gamma$ coincides with the geodesic distance) and, more broadly, for strongly local regular Dirichlet forms where $d_\Gamma$ yields the natural geometry of the diffusion. In such cases, a curvature lower bound $CD(\rho,\infty)$ directly translates into exponential contraction in $W_1$ for the intrinsic metric.

We also prove that a contraction in the Lipschitz seminorm implies contraction of oscillations.

\begin{lemma}[Lipschitz Contraction Implies Oscillation Contraction]\label{lem:lip-to-osc}
Assume \eqref{eq:diam-ass}. Then, for every $f\in\mathrm{Lip}_b(E)$,
\[
\osc(f)\le D \mathrm{Lip}(f).
\]
In particular, under Assumption~\ref{ass:UEWC}, for every $\theta\in\Theta$ and $t\ge0$,
\[
\osc(P_t^\theta f)\le D e^{-\lambda t}\mathrm{Lip}(f),\qquad f\in\mathrm{Lip}_b(E).
\]
\end{lemma}

\begin{proof}
For any $x,y\in E$, $|f(x)-f(y)|\le \mathrm{Lip}(f) d(x,y)\le D \mathrm{Lip}(f)$. Taking $\sup_{x,y}$ gives
$\osc(f)\le D \mathrm{Lip}(f)$. The second claim follows from Lemma~\ref{lem:wass-lip-equiv}.
\end{proof}

We now turn to the two complementary extensions that avoid compactness of $\mathbf{Q}$. First, in Section \ref{static-env}, we analyze \emph{static} upper/lower
envelopes (optimizing only at the terminal time), which yield ergodic limits without any attainment. That is, without needing the supremum/infimum to be achieved by some $\theta^\star \in \Theta$, or equivalently any ``extremal'' semigroup/model.

Second, in Section \ref{nisio-section} we
construct \emph{time-consistent} (Nisio) nonlinear semigroups by allowing switching, thus going beyond Definition~\ref{cpct-imsg-def}  which explicitly excludes switches, and prove their robust ergodicity
under uniform contraction.

In both sections we retain Assumptions \ref{ass:SB-wass} and \ref{ass:UEWC}. By Lemma \ref{lem:BE-implies-UEWC}, the latter can be ensured by asking the elements of $\mathbf{Q}$ to all have positive Bakry-Émery curvature condition. This viewpoint gives the present analysis a geometric flavor, as illustrated for instance by
the Ornstein-Uhlenbeck family in Example~\ref{ex:OU-w1}.

\subsection{Static Envelopes: Robust Limits on $\mathrm{Lip}_b(E)$}\label{static-env}

For $f\in\mathrm{Lip}_b(E)$ define the static upper and lower envelopes
\begin{equation}\label{eq:static-env-wass}
\overline T_t f(x)\coloneqq \sup_{\theta\in\Theta} P_t^\theta f(x),
\qquad
\underline T_t f(x)\coloneqq \inf_{\theta\in\Theta} P_t^\theta f(x),
\qquad t\ge0,
\end{equation}
interpreted as $\mathcal F^\star$-measurable envelopes -- see Appendix \ref{rem:meas}, where the interested reader can find a discussion on the measurability of $x \mapsto \overline{T}_tf(x)$ and $x \mapsto \underline{T}_tf(x)$, and of the switching operators introduced in Definitions \ref{def:nisio-wass} and 
\ref{def:nisio-lower}. 

In Imprecise Probability terminology, given a set $\mathcal{P}$ of probability measures on $(E,\mathcal F)$, the
\emph{upper} and \emph{lower} expectations of a bounded observable $f$ are defined by
\[
\overline{\mathbb E}(f)\coloneqq \sup_{P\in\mathcal P}\int_E f \mathrm d P,
\qquad
\underline{\mathbb E}(f)\coloneqq \inf_{P\in\mathcal P}\int_E f \mathrm d P.
\]
The operators in \eqref{eq:static-env-wass} can be viewed as a \emph{dynamic} analogue of this construction: For fixed $t\ge0$ and
initial state $x\in E$, each candidate model $\theta$ induces the law $\delta_x P_t^\theta$ of $X_t$ under $\mathbf P^\theta$,
so that $P_t^\theta f(x)=\int_E f \mathrm d(\delta_x P_t^\theta)$ is the corresponding (precise) expectation.
Thus, the envelopes
\[
\overline T_t f(x)=\sup_{\theta\in\Theta} P_t^\theta f(x)
=\sup_{\theta\in\Theta}\int_E f \mathrm d(\delta_x P_t^\theta),
\qquad
\underline T_t f(x)=\inf_{\theta\in\Theta}\int_E f \mathrm d(\delta_x P_t^\theta),
\]
are exactly the upper and lower expectations of $f$ with respect to the time-$t$  set
$\mathcal P_t(x)\coloneqq \{\delta_x P_t^\theta: \theta\in\Theta\}$ of one-time marginals.

We first prove a model-wise ergodic estimate that will later be used to identify the upper and lower bounds for the long-term behavior of the envelopes associated with the imprecise Markov semigroup $\mathbf{Q}$.

\begin{proposition}[Unique Invariant Law and Uniform Ergodic Bound in $W_1$]\label{prop:unique-inv-wass}
Let Assumptions~\ref{ass:SB-wass} and \ref{ass:UEWC} hold, and assume \eqref{eq:diam-ass}. Fix $\theta\in\Theta$.
Then, there exists a unique invariant probability measure $\mu_\theta\in\Delta_E$ such that for all $t\ge0$,
\begin{equation}\label{eq:w1-erg}
W_1(\delta_x P_t^\theta,\mu_\theta)\le D e^{-\lambda t},\qquad \forall x\in E.
\end{equation}
Consequently, for every $f\in\mathrm{Lip}_b(E)$,
\begin{equation}\label{eq:unif-erg-lip}
\sup_{x\in E}\bigl|P_t^\theta f(x)-\mu_\theta(f)\bigr|
\le D e^{-\lambda t}\mathrm{Lip}(f),
\end{equation}
where $\mu_\theta(f) \coloneqq \int_E f \text{d} \mu_\theta$.
\end{proposition}

\begin{proof}
Fix $\theta$ and $t_0>0$. By \eqref{eq:UEWC}, the map $\Phi:\Delta_E^1\to\Delta_E^1$,
$\Phi(\nu)=\nu P_{t_0}^\theta$, is a strict contraction in $W_1$ with constant $e^{-\lambda t_0}<1$.
Since $(\Delta_E^1,W_1)$ is complete and $\Delta_E^1=\Delta_E$ under \eqref{eq:diam-ass}, Banach's fixed point
theorem yields a unique fixed point $\mu_\theta$ and $\nu P_{n t_0}^\theta\to\mu_\theta$ in $W_1$.

Notice that invariance extends to all $t\ge0$:
for any $s\ge0$, $\mu_\theta P_s^\theta$ is also a fixed point of $\Phi$ (by the semigroup property), hence equals
$\mu_\theta$ by uniqueness.

For \eqref{eq:w1-erg}, by \eqref{eq:UEWC} and invariance,
\[
W_1(\delta_x P_t^\theta,\mu_\theta)
= W_1(\delta_x P_t^\theta,\mu_\theta P_t^\theta)
\le e^{-\lambda t} W_1(\delta_x,\mu_\theta).
\]
Under \eqref{eq:diam-ass}, $W_1(\delta_x,\mu_\theta)=\int_E d(x,y)\mu_\theta(\mathrm dy)\le D$, yielding
\eqref{eq:w1-erg}. Finally, \eqref{eq:unif-erg-lip} follows from \eqref{eq:KR},
\[
|P_t^\theta f(x)-\mu_\theta(f)|
=\left|\int_E f \mathrm d(\delta_x P_t^\theta-\mu_\theta)\right|
\le \mathrm{Lip}(f) W_1(\delta_x P_t^\theta,\mu_\theta)
\le D e^{-\lambda t}\mathrm{Lip}(f).
\]
\end{proof}

We are now ready to study the ergodicity of $\mathbf{Q}$.

\begin{theorem}[Static Envelope Limits Under Wasserstein Contraction]\label{thm:static-limits-wass}
Let Assumptions~\ref{ass:SB-wass} and \ref{ass:UEWC} hold, and assume \eqref{eq:diam-ass}. Then for every
$f\in\mathrm{Lip}_b(E)$,
\begin{align}
\lim_{t\to\infty}\overline T_t f(x) &= \sup_{\theta\in\Theta}\mu_\theta(f),\label{eq:lim-upper-wass}\\
\lim_{t\to\infty}\underline T_t f(x) &= \inf_{\theta\in\Theta}\mu_\theta(f),\label{eq:lim-lower-wass}
\end{align}
and the convergence is uniform in $x$. Quantitatively,
\[
\sup_{x\in E}\Bigl|\overline T_t f(x)-\sup_{\theta\in\Theta}\mu_\theta(f)\Bigr|
\le D e^{-\lambda t}\mathrm{Lip}(f),
\qquad
\sup_{x\in E}\Bigl|\underline T_t f(x)-\inf_{\theta\in\Theta}\mu_\theta(f)\Bigr|
\le D e^{-\lambda t}\mathrm{Lip}(f).
\]
\end{theorem}

\begin{proof}
By Proposition~\ref{prop:unique-inv-wass}, for every $\theta$ and $x\in E$,
\[
P_t^\theta f(x)=\mu_\theta(f)+\varepsilon_{\theta,t}(x),
\qquad
|\varepsilon_{\theta,t}(x)|\le D e^{-\lambda t}\mathrm{Lip}(f).
\]
Taking the supremum (resp. infimum) over $\theta$ yields the desired bounds and limits.
\end{proof}


Similarly to Corollaries \ref{limiting-cor} and \ref{limiting-cor2}, the following is an important result that subsumes the ergodic behavior of the elements of an IMSG $\mathbf{Q}$ that is not assumed to be compact, and that holds for all bounded Lipschitz functions on a finite-diameter Polish metric state space $E$. It follows immediately from Theorem \ref{thm:static-limits-wass}.

\begin{corollary}[Pointwise Ergodic Limit for Each Model]\label{cor:pointwise-erg-wass}
Assume the hypotheses of Proposition~\ref{prop:unique-inv-wass}. Then, for every
$f\in\mathrm{Lip}_b(E)$ and every $\theta\in\Theta$,
\[
\lim_{t\to\infty} P_t^\theta f(x)=\mu_\theta(f)\qquad \text{uniformly in }x\in E.
\]
In particular, for all $\theta\in\Theta$ and all $x\in E$,
\[
\lim_{t\to\infty} P_t^\theta f(x)\in
\left[\inf_{\theta\in\Theta}\mu_\theta(f), \sup_{\theta\in\Theta}\mu_\theta(f)\right].
\]
\end{corollary}


\subsection{A Time-Consistent Sublinear Semigroup (Nisio) and Robust Ergodicity}\label{nisio-section}

Notice that the static envelopes $\overline T_t$ in Section \ref{static-env} need not satisfy the semigroup property. That is, it need not hold that for all $t,s\geq 0$, $\overline T_{t+s} = \overline T_t \circ \overline T_s$.\footnote{The same holds for the static lower envelopes $\underline T_t$.} We now build a time-consistent sublinear semigroup by switching along partitions, which satisfies the semigroup property.

\begin{definition}[Time-Consistent Switching, Nisio Semigroup]\label{def:nisio-wass}
For $t\ge0$ and $f\in\mathrm{Lip}_b(E)$, define the one-step upper envelope $\overline T_t f$ as in
\eqref{eq:static-env-wass}. Let $\Pi_t$ be the set of all finite partitions $\pi=(t_0,\dots,t_n)$ of $[0,t]$
with $0=t_0<\cdots<t_n=t$. For $\pi=(t_0,\dots,t_n)\in\Pi_t$, define
\[
T_t^\pi f \coloneqq \overline T_{t_n-t_{n-1}}\circ\cdots\circ \overline T_{t_1-t_0} f,
\]
and define the pasted semigroup $(T_t)_{t\ge0}$ by the envelope
\begin{equation}\label{eq:Tt-def-wass}
T_t f \coloneqq \sup_{\pi\in\Pi_t} T_t^\pi f,\qquad t\ge0,
\end{equation}
interpreted as a universally measurable envelope.
\end{definition}

The static envelopes $\overline T_t f(x)=\sup_{\theta\in\Theta}P_t^\theta f(x)$ and $\underline T_t f(x)=\inf_{\theta\in\Theta}P_t^\theta f(x)$
should be viewed as the \emph{one-step} upper and lower expectations of the observable $f(X_t)$ under a  family of models:
Starting from $X_0=x$, the set of candidate one-time marginals at time $t$ is
$\{\delta_x P_t^\theta:\theta\in\Theta\}$, so $\overline T_t f(x)$ (resp. $\underline T_t f(x)$) is the upper (resp. lower) expectation of $f$
over that set. The Nisio (switching) semigroup $(T_t)_{t\ge0}$ refines this viewpoint by allowing \emph{time-consistent} recombination of models:
Informally, $T_t f(x)$ corresponds to optimizing over \emph{adapted switching rules} that select, over successive time intervals, which semigroup
$P^\theta$ governs the evolution, yielding a dynamic programming principle $T_{t+s}=T_t\circ T_s$. In this sense, $(T_t)$ plays the role of a
\emph{dynamic} upper expectation (and its lower analogue of a dynamic lower expectation), exactly as in robust control.

This dynamic upper/lower expectation perspective is closely related to {Choquet} capacities: If we encode ambiguity at time $t$
by set functions on events,
\[
\overline{\mathsf P}_t(A\mid x)\coloneqq \sup_{\theta\in\Theta} (\delta_x P_t^\theta)(A),
\qquad
\underline{\mathsf P}_t(A\mid x)\coloneqq \inf_{\theta\in\Theta} (\delta_x P_t^\theta)(A),
\qquad A\in\mathcal F,
\]
then for bounded measurable $f$ the corresponding upper/lower expectations can be expressed via the associated (upper/lower) Choquet expectations.
In general, however, the switching/Nisio operators $T_t$ and their lower counterparts are best understood as \emph{time-consistent} (sub/super)linear
envelopes of conditional expectations (given by iterated one-step suprema/infima along partitions), rather than as integration with respect to a
single fixed capacity on path space; Allowing switches typically leads to a nonlinear, dynamically consistent evaluation that subsumes (and may be
strictly richer than) the one-step Choquet upper/lower expectations at fixed horizons.

The Nisio Semigroup satisfies the semigroup property.

\begin{remark}[Semigroup Property]\label{rem:nisio-dpp-wass}
Under Assumption~\ref{ass:SB-wass}, the family $(T_t)_{t\ge0}$ is the (upper) Nisio semigroup associated with
$(\overline T_t)_{t\ge0}$ and satisfies the semigroup property (sometimes referred to as the dynamic programming principle),
\[
T_{t+s}=T_t\circ T_s,\qquad \forall t,s\ge0.
\]
To see this, we refer the reader to standard treatments of Nisio semigroups/stochastic control such as
\citet[Chapter~III]{fleming_soner}.
\end{remark}

The switching construction yields a \emph{time-consistent} nonlinear evolution: It satisfies the dynamic programming principle
and acts as a coherent (sublinear) upper expectation on test functions. We summarize the resulting semigroup and
regularity properties in the next proposition.

\begin{proposition}[Sublinear Markov Semigroup on $\mathrm{Lip}_b(E)$]\label{prop:sublinear-semigroup-wass}
The family $\mathbf T=(T_t)_{t\ge0}$ of Definition~\ref{def:nisio-wass} satisfies
\begin{itemize}
\item[(i)] $T_0=\mathrm{Id}$ on $\mathrm{Lip}_b(E)$;
\item[(ii)] $T_{t+s}=T_t\circ T_s$ for all $t,s\ge0$;
\item[(iii)] Each $T_t$ is monotone, constant-preserving, sublinear, and translation invariant:
$T_t(f+c)=T_t f+c$ for all $c\in\mathbb R$;
\item[(iv)] $\|T_t f-T_t g\|_\infty\le \|f-g\|_\infty$ for all $f,g \in \mathrm{Lip}_b(E)$.
\end{itemize}
\end{proposition}

\begin{proof}
(i) For $t=0$, $\Pi_0$ contains only $\pi=(0)$, hence $T_0^\pi f=f$ and $T_0 f=f$.

(ii) This is the semigroup property, see Remark~\ref{rem:nisio-dpp-wass}.

(iii) Each $\overline T_u$ is monotone, constant-preserving, sublinear and translation invariant (as a pointwise
supremum of linear Markov operators). These properties are preserved under composition (thus for $T_t^\pi$) and
under pointwise suprema over $\pi$, hence they hold for $T_t$.

(iv) If $\|f-g\|_\infty\le \varepsilon$, then $g-\varepsilon\le f\le g+\varepsilon$. By monotonicity and translation
invariance, $T_t g-\varepsilon\le T_t f\le T_t g+\varepsilon$, yielding the claim.
\end{proof}

To obtain long-term stability and an ergodic limit for the Nisio semigroup, we need a quantitative regularization property.
The next lemma shows that the one-step Lipschitz/Wasserstein contraction propagates through switching, yielding the same
contraction (hence oscillation decay) for $(T_t)_{t\ge0}$. As a byproduct, the next lemma shows also that $T_t$ maps $\mathrm{Lip}_b(E)$ into $\mathrm{Lip}_b(E)$, hence $T_t f$ is in fact Borel measurable.

\begin{lemma}[Lipschitz contraction for $\overline T_t$ and $T_t$]\label{lem:nisio-lip}
Let Assumption~\ref{ass:UEWC} hold. Then, for all $t\ge0$ and $f\in\mathrm{Lip}_b(E)$,
\begin{equation}\label{eq:T-lip}
\mathrm{Lip}(\overline T_t f)\le e^{-\lambda t}\mathrm{Lip}(f),
\qquad
\mathrm{Lip}(T_t f)\le e^{-\lambda t}\mathrm{Lip}(f).
\end{equation}
Consequently, under \eqref{eq:diam-ass},
\begin{equation}\label{eq:T-osc-wass}
\osc(T_t f)\le D e^{-\lambda t}\mathrm{Lip}(f).
\end{equation}
\end{lemma}

\begin{proof}
Fix $t\ge0$ and $f\in\mathrm{Lip}_b(E)$. For each $\theta$, Lemma~\ref{lem:wass-lip-equiv} gives
$\mathrm{Lip}(P_t^\theta f)\le e^{-\lambda t}\mathrm{Lip}(f)$. Since the pointwise supremum of $L$-Lipschitz functions is
$L$-Lipschitz,\footnote{Indeed, for any $x,y\in E$, $\sup_i u_i(x)-\sup_i u_i(y)\le \sup_i\bigl(u_i(x)-u_i(y)\bigr)\le L d(x,y)$, and swapping $x,y$ gives $|\sup_i u_i(x)-\sup_i u_i(y)|\le L d(x,y)$.
} $\overline T_t f=\sup_\theta P_t^\theta f$ satisfies
$\mathrm{Lip}(\overline T_t f)\le e^{-\lambda t}\mathrm{Lip}(f)$. 

For $T_t$, fix $\pi=(t_0,\dots,t_n)\in\Pi_t$.
Iterating the previous bound along the composition defining $T_t^\pi$ yields
$\mathrm{Lip}(T_t^\pi f)\le e^{-\lambda t}\mathrm{Lip}(f)$. Taking the supremum over $\pi$ preserves the Lipschitz constant, so
$\mathrm{Lip}(T_t f)\le e^{-\lambda t}\mathrm{Lip}(f)$. Finally, \eqref{eq:T-osc-wass} follows from Lemma~\ref{lem:lip-to-osc}.
\end{proof}

We are now in a position to state the first main conclusion of this section: The time-consistent sublinear (Nisio) semigroup
inherits the uniform Wasserstein/Lipschitz contraction and therefore admits a \emph{robust} ergodic limit. In particular,
for every bounded Lipschitz test functional $f$, the worst-case value function $T_t f$ converges uniformly to a constant (a sublinear
``equilibrium expectation''), independently of the initial state.

\begin{theorem}[Robust Ergodicity of the Nisio Semigroup Under Wasserstein Contraction]\label{thm:robust-erg-wass}
Let Assumptions~\ref{ass:SB-wass} and \ref{ass:UEWC} hold, and assume \eqref{eq:diam-ass}. Then for every
$f\in\mathrm{Lip}_b(E)$ there exists a constant $\mathcal R(f)\in\mathbb R$ such that
\[
\lim_{t\to\infty}\|T_t f-\mathcal R(f)\|_\infty=0.
\]
Moreover, fixing $t_0>0$ and setting $\kappa\coloneqq e^{-\lambda t_0}\in(0,1)$, we have that, for all $t\ge0$,
\[
\|T_t f-\mathcal R(f)\|_\infty \le C_{t_0} D e^{-\lambda t}\mathrm{Lip}(f),
\qquad
C_{t_0}\coloneqq e^{\lambda t_0}\frac{2-\kappa}{1-\kappa}.
\]
Finally, $\mathcal R$ is monotone, constant-preserving, sublinear, and satisfies $\mathcal R(T_t f)=\mathcal R(f)$,
for all $t\ge0$.
\end{theorem}

\begin{proof}
Fix $f\in\mathrm{Lip}_b(E)$ and $t_0>0$. Set $S\coloneqq T_{t_0}$ and $\kappa=e^{-\lambda t_0}$.
By Lemma~\ref{lem:nisio-lip}, $\osc(S^n f)\le D \mathrm{Lip}(S^n f)\le D \kappa^n\mathrm{Lip}(f)$.

Pick $x_0\in E$ and define $a_n\coloneqq S^n f(x_0)$. Since $S$ is monotone and constant-preserving,
$\inf g\le Sg\le \sup g$ pointwise for every bounded $g$. Applying this to $g=S^n f$ yields that both $a_n$
and $a_{n+1}$ lie in the interval $[\inf S^n f,\sup S^n f]$; Hence,
\[
|a_{n+1}-a_n|\le \osc(S^n f)\le D \kappa^n\mathrm{Lip}(f).
\]
Thus $(a_n)$ is Cauchy; Define $\mathcal R(f)\coloneqq \lim_{n\to\infty} a_n$.

For any $x\in E$, $|S^n f(x)-a_n|\le \osc(S^n f)\le D \kappa^n\mathrm{Lip}(f)$, and therefore
\[
\|S^n f-\mathcal R(f)\|_\infty
\le D \kappa^n\mathrm{Lip}(f)+\sum_{k=n}^\infty |a_{k+1}-a_k|
\le D \kappa^n\mathrm{Lip}(f)+\sum_{k=n}^\infty D \kappa^k\mathrm{Lip}(f)
= D \frac{2-\kappa}{1-\kappa} \kappa^n\mathrm{Lip}(f).
\]
For general $t\ge0$, write $t=n t_0+r$ with $r\in[0,t_0)$. By the semigroup property,
$T_t f=T_r(S^n f)$. Using nonexpansiveness in $\|\cdot\|_\infty$ and constant preservation,
\[
\|T_t f-\mathcal R(f)\|_\infty
=\|T_r(S^n f)-T_r(\mathcal R(f))\|_\infty
\le \|S^n f-\mathcal R(f)\|_\infty
\le D \frac{2-\kappa}{1-\kappa} \kappa^n\mathrm{Lip}(f)
\le C_{t_0} D e^{-\lambda t}\mathrm{Lip}(f).
\]
The remaining properties of $\mathcal R$ follow by taking limits of the maps $f\mapsto S^n f(x_0)$ and using
the corresponding properties of $S$ from Proposition~\ref{prop:sublinear-semigroup-wass}.
\end{proof}

Of course, we have a similar result for the conjugate of the Nisio semigroup, a superlinear semigroup acting as a lower expectation.

\begin{definition}[Lower Nisio Semigroup]\label{def:nisio-lower}
Let $(T_t)_{t\ge 0}$ be the upper Nisio semigroup built from the one-step envelopes
$\overline T_t f=\sup_{\theta\in\Theta}P_t^\theta f$. Define the \emph{lower Nisio semigroup}
$(T_t^{-})_{t\ge 0}$ on $\mathrm{Lip}_b(E)$ by
\[
T_t^{-} f  \coloneqq  -T_t(-f),\qquad t\ge 0.
\]
Equivalently, with the one-step lower envelope $\underline T_t f\coloneqq \inf_{\theta\in\Theta}P_t^\theta f$,
we have
\[
T_t^{-} f  =  \inf_{\pi\in\Pi_t}\Bigl(\underline T_{t_n-t_{n-1}}\circ\cdots\circ
\underline T_{t_1-t_0} f\Bigr),
\]
interpreted as a universally measurable envelope.
\end{definition}

\begin{theorem}[Robust Ergodicity of the Lower Nisio Semigroup]\label{thm:robust-erg-wass-lower}
Let Assumptions~\ref{ass:SB-wass} and \ref{ass:UEWC} hold, and assume \eqref{eq:diam-ass}. Then, for every
$f\in\mathrm{Lip}_b(E)$ there exists a constant $\mathcal{R}^{-}(f)\in\mathbb R$ such that
\[
\lim_{t\to\infty}\|T_t^{-} f-\mathcal{R}^{-}(f)\|_\infty=0.
\]
Moreover, fixing $t_0>0$ and setting $\kappa\coloneqq e^{-\lambda t_0}\in(0,1)$, we have that for all $t\ge 0$,
\[
\|T_t^{-} f-\mathcal{R}^{-}(f)\|_\infty
\le C_{t_0}De^{-\lambda t}\mathrm{Lip}(f),
\qquad
C_{t_0}\coloneqq e^{\lambda t_0}\frac{2-\kappa}{1-\kappa}.
\]
Finally, $\mathcal{R}^{-}$ is monotone, constant-preserving, \emph{superlinear} (i.e. concave):
\[
\mathcal{R}^{-}(f+g)\ge \mathcal{R}^{-}(f)+\mathcal{R}^{-}(g),\qquad
\mathcal{R}^{-}(\alpha f)=\alpha\mathcal{R}^{-}(f)  (\alpha\ge 0),
\]
it is translation invariant $\mathcal{R}^{-}(f+c)=\mathcal{R}^{-}(f)+c$, and it satisfies
$\mathcal{R}^{-}(T_t^{-} f)=\mathcal{R}^{-}(f)$ for all $t\ge 0$.
\end{theorem}

\begin{proof}
Let $\mathcal{R}$ be the limit functional from Theorem~\ref{thm:robust-erg-wass}. By definition,
\[
T_t^{-}f=-T_t(-f).
\]
Applying Theorem~\ref{thm:robust-erg-wass} to $-f$ yields
\[
\|T_t(-f)-\mathcal{R}(-f)\|_\infty\to 0
\]
with the stated rate. Hence,
\[
\|T_t^{-}f-(-\mathcal{R}(-f))\|_\infty \to 0,
\]
with the same bound. Define $\mathcal{R}^{-}(f)\coloneqq -\mathcal{R}(-f)$.
All stated properties of $\mathcal{R}^{-}$ and $(T_t^{-})_{t\ge 0}$ follow by conjugation from the
corresponding sublinear properties of $\mathcal{R}$ and $(T_t)_{t\ge 0}$.
\end{proof}

Theorems \ref{thm:robust-erg-wass} and \ref{thm:robust-erg-wass-lower} give us the following reflection. 

\begin{remark}[Relationship Between Static and Pasted Limits]\label{rem:static-vs-nisio}
Let Assumptions~\ref{ass:SB-wass} and \ref{ass:UEWC} hold, and assume \eqref{eq:diam-ass}.
Consider the upper Nisio  semigroup $(T_t)_{t\ge 0}$ built from the one-step envelopes
$\overline T_t f=\sup_{\theta\in\Theta}P_t^\theta f$ as in Definition~\ref{def:nisio-wass}
(and define analogously the lower Nisio semigroup $(T_t^{-})_{t\ge 0}$ from
$\underline T_t f=\inf_{\theta\in\Theta}P_t^\theta f$).
Then, for every $t\ge 0$ and $f\in\mathrm{Lip}_b(E)$,
\[
T_tf  \ge  \overline T_t f=\sup_{\theta\in\Theta}P_t^\theta f,
\qquad
T_t^{-}f  \le  \underline T_t f=\inf_{\theta\in\Theta}P_t^\theta f.
\]
Indeed, the one-block partition $\pi=(0,t)$ belongs to $\Pi_t$, so $T_t^{\pi}f=\overline T_t f$ and
$T_tf=\sup_{\pi\in\Pi_t}T_t^{\pi}f\ge \overline T_t f$; The lower bound is analogous.

Consequently, if $\mathcal R(f)$ and $\mathcal R^{-}(f)$ denote the ergodic limits from
Theorems~\ref{thm:robust-erg-wass} (upper) and \ref{thm:robust-erg-wass-lower} (lower), respectively, then
\[
\mathcal R(f) \ge \sup_{\theta\in\Theta}\mu_\theta(f),
\qquad
\mathcal R^{-}(f) \le \inf_{\theta\in\Theta}\mu_\theta(f),
\]
where $\mu_\theta$ is the unique invariant law of $\mathbf P^\theta$ from
Proposition~\ref{prop:unique-inv-wass}.
In general, these inequalities may be strict, reflecting the additional flexibility of
state-dependent switching allowed by the pasted (time-consistent) construction.
\end{remark}

\section{Conclusion}\label{concl}

In this work, we introduced the concept of an imprecise Markov semigroup $\mathbf{Q}$, and studied its ergodic behavior  (i) when the state space $E$ is a Euclidean space or a Riemannian manifold and $\mathbf{Q}$ is compact; (ii) when $E$ is allowed to be an arbitrary good measurable space and $\mathbf{Q}$ is compact; and (iii) when $E$ is a generic Polish metric space with finite diameter, and  $\mathbf{Q}$ is not assumed compact. 


In future work, we plan to extend the results of Section \ref{subsec:wass-sublinear} beyond the function class $\text{Lip}_b(E)$ and to relax the finite-diameter assumption \eqref{eq:diam-ass}. We also plan to extend our theory to explicitly model ambiguity in the initial distribution of a Markov process.

Furthermore, we aim to apply these results to machine learning and computer vision, thereby making the high-level discussion in the Introduction more concrete, and to develop robust reinforcement learning methods, where Markov processes are used for planning by autonomous agents.

We are also interested in extending to the imprecise setting the types of inequalities developed in \citet[Chapters 6-9]{bakry}. In particular, we are interested in deriving imprecise optimal transport inequalities, in the spirit of \citep{caprio2024optimaltransportepsiloncontaminatedcredal,ipmu}, for application to distribution shift problems in machine learning and computer vision \citep{vivian}.


\section*{Acknowledgements}
We would like to acknowledge partial funding by the Army Research Office, grant ARO MURI W911NF2010080. We wish to express our gratitude to Mira Gordin for insightful discussions, particularly on how the gradient bound for a diffusion operator relates to Poincaré inequalities, and to Simone Cerreia-Vioglio and Paolo Perrone for their help in finding the suitable topology compatible with the partial order $\preceq^\text{part}_{\tilde f}$, and in shaping up the discussion after Lemma \ref{uniqueness}.
We are also extremely grateful to Enrique Miranda for his suggestion to look into \citet[Example 6.2]{krak}, to Andrea Aveni for his comments on an initial version of this work, and to two anonymous reviewers for their invaluable remark that greatly improved our manuscript. We are also indebted to Malena Español for her invitation to the Institute for Advanced Study, where the idea of the present paper germinated. 

\appendix

\section{Proof of Theorem \ref{thm-main-2}}\label{app-a}
    We begin by pointing out that, as a consequence of Lemma \ref{uniqueness} and Definition \ref{idmt-def}, $\underline{\mathbf{P}} $ and $\overline{\mathbf{P}} $ are (precise) Markov semigroups, and so the operators $\underline{\Gamma} $, $\overline{\Gamma} $, $\underline{L} $ and $\overline{L} $ are well-defined. 

    We first focus on the diffusion Markov triple $(E,\overline{\mu} ,\overline{\Gamma} )$. By \citet[Proposition 3.3.11]{bakry} and  $\overline{\mathbf{P}} =(\overline{P} _t)_{t \geq 0}$ being a (precise) Markov semigroup, the fact that $\overline{L} $ is essentially self-adjoint follows from the connexity, weak hypo-ellipticity, and completeness assumptions. Pick any $f\in\mathcal{A}_0^{\overline{\mu}}$. The fact that $\overline{L} $ satisfies ESA, together with our assumption that it satisfies the curvature condition $CD(\rho_{\overline{L}},\infty)$, for some $\rho_{\overline{L}} \in\mathbb R$, implies -- by \citet[Corollary 3.3.19]{bakry} -- that for all $t \geq 0$, the following gradient bound holds
    \begin{equation}\label{eq-grad-bd-gen}
        \overline{\Gamma} \left(\overline{P} _t f \right) \leq e^{-2t \rho_{\overline{L}}} \overline{P} _t\left( \overline{\Gamma} (f) \right).
    \end{equation}
    Then, by \citet[Theorem 4.7.2]{bakry}, equation \eqref{eq-grad-bd-gen} implies that the following local Poincaré inequality holds for all $t\geq 0$,
    \begin{equation}\label{local-poinc-ineq-gen}
        \overline{P} _t(f^2) - \left( \overline{P} _t f \right)^2 \leq \frac{1-e^{-2t \rho_{\overline{L}}}}{\rho_{\overline{L}}} \overline{P} _t\left( \overline{\Gamma} (f) \right).
    \end{equation}
    Here the factor needs to be understood as $2t$ when $\rho_{\overline{L}}=0$.
Following \citet[Theorem 2.35]{handel}, 
    taking the limit as $t\rightarrow \infty$ on both sides of  \eqref{local-poinc-ineq-gen} yields the following Poincaré inequality
    \begin{equation}\label{poinc-ineq-gen}
        \text{Var}_{\overline{\mu} }(f) \leq \frac{1}{\rho_{\overline{L}}} \overline{\mathcal{E}} (f).
    \end{equation}
    As in the proof of Theorem \ref{thm-main-1}, here it is implicitly assumed that $\rho_{\overline{L}}>0$. In addition, 
    we know that 
    \begin{equation}\label{int-by-parts-gen}
        \overline{\mathcal{E}} (f)=-\int_E f \overline{L}  f \text{d} \overline{\mu} .
    \end{equation}
    Combining \eqref{int-by-parts-gen} with \eqref{poinc-ineq-gen}, we have that if function $f$ is such that $\overline{L}  f =0$, then $\overline{\mathcal{E}} (f)=0$. In turn, this implies that $\text{Var}_{\overline{\mu} }(f)=0$, and so the function $f$ must be constant $\overline{\mu} $-a.e. Notice now that, since $\overline{L} $ satisfies the ESA property, $\mathcal{A}_0^{\overline{\mu}}$ is dense in $\mathcal{D}(\overline{\mathcal{E}} )$ in the topology induced by the Dirichlet norm $\|\cdot \|_{\overline{\mathcal{E}} }$.\footnote{This norm is recalled in Definition \ref{diff-mark-trip}.(DMT4) in \ref{diff-mak-tr-back}.}  Thanks to this, \eqref{eq-grad-bd-gen}-\eqref{int-by-parts-gen} hold for all $f\in\mathcal{D}(\overline{\mathcal{E}} )$ by an approximation argument. This implies that every $f\in\mathcal{D}(\overline{\mathcal{E}} )$ such that $\overline{\Gamma} (f)=0$ is constant, and so that $\overline{L} $ is ergodic according to \citet[Definition 3.1.11]{bakry}. Then, since $\overline{\mu} $ is finite, by \citet[Proposition 3.1.13]{bakry} we have that
    \begin{equation*}
        \forall f \in \mathbb{L}^2(\overline{\mu} ), \quad \lim_{t \rightarrow \infty} \overline{P}_t  f =\int_E f \text{d}\overline{\mu} , \quad \overline{\mu} \text{-a.e.}
    \end{equation*}
    in $\mathbb{L}^2(\overline{\mu} )$. This concludes the first part of the proof. The second part, that is, the one pertaining the least element $\underline{\mathbf{P}} $ of $\mathbf{Q}$, is analogous. The last statement of the theorem is immediate once we let $\overline{\mu} =\underline{\mu} =\mu $. \hfill \qedsymbol

\section{Measurability of Envelopes and Switching Operators}\label{rem:meas}
Under Assumption~\ref{ass:SB-wass}, for each $t\ge0$ and bounded {Borel} $f$, the map
$(x,\theta)\mapsto P_t^\theta f(x)$ is Borel measurable (by a standard monotone class argument, starting from indicators $\mathbbm{1}_A$). Now, let us denote by $\mathcal F^\star$  the \emph{universal completion} of $\mathcal F$, and  by
$B^\star(E)$ the space of bounded $\mathcal F^\star$-measurable real-valued functionals on $E$.
Since $f$ is $\mathcal F^\star$-measurable,
it is measurable with respect to the completion of $p_t^\theta(x,\cdot)$ for every $(t,x,\theta)$; Hence we
may extend $P_t^\theta$ to $B^\star(E)$ by
\[
P_t^\theta f(x)\coloneqq \int_E f(y) p_t^\theta(x,\mathrm dy),
\qquad f\in B^\star(E),
\]
which is well-defined and $\mathcal F^\star$-measurable in $x$.

Consequently, for each fixed $t\ge0$ and bounded {Borel} $f$, the envelope
$x\mapsto \sup_{\theta\in\Theta}P_t^\theta f(x)$ is {upper semianalytic}, hence $\mathcal F^\star$-measurable.
Likewise, the \emph{time-consistent} switching operator introduced in Definition~\ref{def:nisio-wass}
is obtained by composing the one-step envelopes $\overline T$ along a finite partition of $[0,t]$
and then taking a supremum over partitions.\footnote{Similarly for the lower Nisio semigroup in Definition \ref{def:nisio-lower}.} Since upper semianalytic functions are stable under
application of Borel kernels -- i.e. $g \mapsto (x \mapsto \int g(y)p(x,\text{d}y))$ -- and under pointwise suprema, for bounded Borel $f$ each $T_t f$ is
upper semianalytic and therefore $\mathcal F^\star$-measurable. For general $f\in B^\star(E)$ the
same conclusion holds at the level of universal measurability, so all operators are interpreted as acting on
$B^\star(E)$.

\section{Background Notions}\label{app-b}

In an effort to keep the paper self-contained, we present in this section the preliminary concepts needed to appreciate the main results in our work.
\subsection{Background on Markov Semigroups}\label{back-MSG}
Let us briefly recall what a (precise) Markov process is; To do so, we borrow the notation and terminology from \citet[Chapter 1.1]{bakry}.
Let $(X_t)_{t \geq 0}$ be a measurable process on a probability space $(\Omega,\Sigma,\mathbb{P})$, such that $X_t \equiv X_t(\omega)$ is an element of the good measurable state space $E$ from Definition \ref{good}, for all $t \geq 0$ and all $\omega\in\Omega$.
Denote by $\mathcal{F}_t \coloneqq \sigma(X_u : u \leq t)$, $t \geq 0$, the natural filtration of $(X_t)_{t \geq 0}$. 

The \textit{Markov property} indicates that for $t>s$, the law of $X_t$ given $\mathcal{F}_s$ is the law of $X_t$ given $X_s$, as well as the law of $X_{t-s}$ given $X_0$, the latter property reflecting the fact that the Markov process is time homogeneous, which is the only case that we consider in the present work. As a consequence, throughout the paper we focus our attention on the law of $X_t$ given $X_0$, for a generic $t \geq 0$.

For any $t\geq 0$, we can describe the distribution of $X_t$ starting from $X_0=x$ via a \textit{probability kernel} $p_t(x,A)$, that is, a function $p_t:E \times \mathcal{F} \rightarrow [0,1]$ such that (i) $p_t(x,\cdot)$ is a probability measure, for all $x\in E$, and (ii) $x \mapsto p_t(x,A)$ is measurable, for every measurable set $A\in\mathcal{F}$.

We now introduce one of the main concepts that we use in the present paper, namely a \textit{Markov semigroup}. We follow the presentation in \cite[Chapter 1.2]{bakry}. Let $B(E)$ denote the set of bounded measurable functionals on $(E,\mathcal{F})$. A Markov semigroup $\mathbf{P}=(P_t)_{t \geq 0}$ is a family of operators that satisfies the following properties.

\begin{itemize}
    \item[(i)] For every $t \geq 0$, $P_t: B(E) \rightarrow B(E)$ is a linear operator (linearity).
    \item[(ii)] $P_0=\text{Id}$, the identity operator (initial condition).
    \item[(iii)] $P_t \mathbbm{1}=\mathbbm{1}$, where $\mathbbm{1}$ is the constant function equal to $1$ (mass conservation).\footnote{In general, we write $P_t(f) \equiv P_t f$, for all $f \in B(E)$, for notational convenience.} 
    \item[(iv)] For every $t \geq 0$, $f \geq 0 \implies P_tf \geq 0$ (positivity preserving).
    \item[(v)] For every $t,s \geq 0$, $P_{t+s}=P_t \circ P_s$ (semigroup property).
\end{itemize}

One extra condition -- continuity at $t = 0$ -- is needed to fully define a Markov semigroup $\mathbf{P}=(P_t)_{t \geq 0}$. To present it though, we first need to introduce the concept of invariant (or stationary) measure.

Let $\mathbf{P}=(P_t)_{t \geq 0}$ be a family of operators satisfying (i)-(v). Then, a (positive) $\sigma$-finite measure $\mu$ on $(E,\mathcal{F})$ is called \textit{invariant} for $\mathbf{P}$ if for every bounded positive measurable functional $f$ on $E$, and every $t \geq 0$, 

\begin{equation}\label{invaraint}
    \int_E P_t f \text{d}\mu=\int_E f \text{d}\mu.
\end{equation}

Notice that, as a consequence of (iii), (iv), and Jensen's inequality, we have that if a measure $\mu$ is invariant for $\mathbf{P}$, then, for every $t \geq 0$, $P_t$ is a contraction on the bounded functions in $\mathbb{L}^p(\mu)$, for any $1 \leq p \leq \infty$.

We are now ready for the complete definition.

\begin{definition}[Markov Semigroup]\label{msg}
    Let $(E,\mathcal{F})$ be a good measurable space. A family $\mathbf{P}=(P_t)_{t \geq 0}$ of operators $P_t$ is called a Markov semigroup if it satisfies (i)-(v), there exists an invariant measure $\mu$ for $\mathbf{P}$, and the following extra condition is satisfied
    \begin{itemize}
        \item[(vi)] For every $f \in \mathbb{L}^2(\mu)$, $\lim_{t \rightarrow 0} P_tf=f$ in $\mathbb{L}^2(\mu)$ (continuity at $t = 0$). 
    \end{itemize}
\end{definition}

Notice that, since $(P_t)_{t \geq 0}$ is a contraction (as pointed out after equation \eqref{invaraint}) and a semigroup (by condition (v)), property (vi) expresses that $t \mapsto P_tf$ is continuous in $\mathbb{L}^2(\mu)$ on $\mathbb{R}_+$. In addition, we could have required convergence in $\mathbb{L}^p(\mu)$ for any $p \in [1,\infty)$, but we stick to $p=2$ for simplicity and to align with \citet{bakry}. 
Before moving on, let us add a remark on invariant measures.

\begin{remark}[Spotlight on Invariant Measures]\label{spotlight}
    In general, an invariant measure $\mu$ for a Markov semigroup $\mathbf{P}$ need not be a \textit{probability} measure.\footnote{If it is a finite measure, it is customary to normalize it to a probability measure. It is then usually unique \citep[Page 11]{bakry}.}
    In addition, it need not always exist; Necessary and sufficient conditions for its existence may be found in \citep{ito_paper,shlomo}. That being said, most semigroups of interest do have an invariant measure \citep[Section 1.2.1]{bakry}. A way to derive $\mu$ as a weak limit of a ``reasonable'' initial probability measure $\mu_0$ is inspected in \citet[Page 10]{bakry}. A typical example of a Markov semigroup having an invariant measure is the heat or Brownian semigroup on $\mathbb{R}^n$, whose invariant measure is (up to multiplication by a constant) the Lebesgue measure.\footnote{The invariant measure is in general only defined up to a multiplicative constant.}
\end{remark}

As pointed out in \citet[Section 1.2.1]{bakry}, when the invariant measure $\mu$ is a probability measure, it has an immediate interpretation. If the Markov process $(X_t)_{t \geq 0}$ starts at time $t=0$ with initial distribution $\mu$, i.e. if $X_0 \sim \mu$, then it keeps the same distribution at each time $t\geq 0$; Indeed,
\begin{align*}
    \mathbb{E}[f(X_t)]=\mathbb{E}[\mathbb{E}(f(X_t) \mid X_0)]=\mathbb{E}[P_t f(X_0)]&=\int_E P_t f \text{d}\mu\\
    &= \int_E f \text{d}\mu = \mathbb{E}[f(X_0)],
\end{align*}
for all $f\in B(E)$ and all $t \geq 0$.  Functionals $f$, then, have to be understood as classes of functions for the $\mu$-a.e. equality, and equalities and inequalities such as $f \leq g$ are always understood to hold $\mu$-a.e.

We now continue our presentation of the concepts related to Markov semigroups that are crucial to derive our main results. An important property of a Markov semigroup is its symmetry \citep[Definition 1.6.1]{bakry}.

\begin{definition}[Symmetric Markov Semigroup]\label{symmetry}
    Let $\mathbf{P}=(P_t)_{t \geq 0}$ be a Markov semigroup with (good) state space $(E,\mathcal{F})$ and invariant measure $\mu$. We say $\mathbf{P}$ is \textit{symmetric} with respect to $\mu$ -- or equivalently, that $\mu$ is \textit{reversible} for $\mathbf{P}$ -- if, for all functions $f,g \in \mathbb{L}^2(\mu)$ and all $t \geq 0$,
    \begin{equation*}
        \int_E f P_t g \text{d}\mu = \int_E g P_t f \text{d}\mu.
    \end{equation*}
\end{definition}

The probabilistic interpretation of Definition \ref{symmetry} is straightforward. As pointed out in \citet[Section 1.6.1]{bakry}, the name ``reversible'' refers to reversibility in time of the Markov process $(X_t)_{t \geq 0}$ associated with the Markov semigroup $\mathbf{P}$ whenever the initial law $\mu$ is the invariant measure. Indeed, from \eqref{invaraint}, we know that if the process starts from the invariant distribution $\mu$, then it keeps the same distribution at each time $t$. Moreover, if the measure is reversible, and if the initial distribution of $X_0$ is $\mu$, then for any $t > 0$ and any partition $0\leq t_1 \leq \cdots \leq t_k\leq t$ of the time interval $[0,t]$, the law of
$(X_0,X_{t_1},\ldots,X_{t_k} ,X_t)$ is the same as the law of $(X_t,X_{t-t_1} ,\ldots,X_{t-t_k} ,X_0)$. Hence, we can say that the law of the Markov process is ``reversible in time''.\footnote{For examples of symmetric Markov semigroups, we refer the reader to \citet[Section 1.6.1, Appendices A.2, A.3]{bakry}.}

We now need to introduce the concept of infinitesimal generator of a Markov semigroup $\mathbf{P}=(P_t)_{t \geq 0}$ \citep[Definition 1.4.1]{bakry}. By Hille-Yosida theory \citep{hille},\footnote{A more modern reference for this result is \cite[Appendix A.1]{bakry}.} there exists a dense linear subspace of $\mathbb{L}^2(\mu)$, called the domain $\mathcal{D}$ of the semigroup $\mathbf{P}$, on which the derivative at $t = 0$ of $P_t$ exists in $\mathbb{L}^2(\mu)$.

\begin{definition}[Infinitesimal Generator]\label{infint-gen}
    Let $\mathbf{P}=(P_t)_{t \geq 0}$ be a Markov semigroup with (good) state space $(E,\mathcal{F})$ and invariant measure $\mu$. The \textit{infinitesimal generator} $L$ of $\mathbf{P}$ in $\mathbb{L}^2(\mu)$ is a linear operator $L:\mathcal{D} \rightarrow \mathbb{L}^2(\mu)$, 
    $$f \mapsto Lf \coloneqq \lim_{t \rightarrow 0} \frac{P_t f - f}{t}.$$
\end{definition}

The domain $\mathcal{D}$ is also called the domain of $L$, denoted by $\mathcal{D}(L)$, and depends on the underlying space $\mathbb{L}^2(\mu)$. Two comments are in place. First, as indicated by \citet[Chapter 1.4]{bakry}, the couple $(L,\mathcal{D}(L))$ completely characterizes the Markov semigroup $\mathbf{P}$ acting on $\mathbb{L}^2(\mu)$. Second, 
the infinitesimal generator $L$ has a natural interpretation. As pointed out by \href{https://danmackinlay.name/notebook/infinitesimal_generators.html}{Dan MacKinlay}, $L$ is a kind of linearization of the local Markov transition kernel for the Markov process $(X_t)_{t \geq 0}$ associated with the Markov semigroup $\mathbf{P}$; This is because $L$ is defined as a derivative.

The infinitesimal generator $L$ of a Markov semigroup $\mathbf{P}$ is sometimes referred to as the \textit{Markov generator} (of $\mathbf{P}$). We can express the reversibility property of Definition \ref{symmetry} in terms of $L$ as

\begin{equation}\label{eq-symm-gen}
    \int_E fLg \text{d}\mu = \int_E gLf \text{d}\mu, \quad \forall f,g \in \mathcal{D}(L).
\end{equation}
The Markov generator is thus symmetric, however it is unbounded in $\mathbb{L}^2(\mu)$.

Next, we present the carré du champ operator, that measures how far an infinitesimal generator $L$ is from being a derivation \citep{ledoux}.\footnote{Here, ``derivation'' refers to the differential algebra meaning of the term \citep{lang}.}

\begin{definition}[Carré du Champ]\label{carre}
    Let $\mathbf{P}=(P_t)_{t \geq 0}$ be a Markov semigroup with (good) state space $(E,\mathcal{F})$ and invariant measure $\mu$. Let $L$ be its infinitesimal generator with $\mathbb{L}^2(\mu)$-domain $\mathcal{D}(L)$. Let $\mathcal{A} \subseteq \mathcal{D}(L)$ be an algebra. Then, the bilinear map
    $$\Gamma: \mathcal{A} \times \mathcal{A} \rightarrow \mathcal{A}, \quad (f,g) \mapsto \Gamma(f,g) \coloneqq \frac{1}{2}\left[ L(fg) -fLg-gLf \right]$$
    is called the \textit{carré du champ operator} of the Markov generator $L$.
\end{definition}

The carré du champ operator is symmetric and is positive on $\mathcal{A}$, that is, $\Gamma(f,f) \geq 0$, for all $f \in \mathcal{A}$.\footnote{This is a fundamental property of Markov semigroups \citep[Section 1.4.2]{bakry}.} In this work, we write $\Gamma(f,f) \equiv \Gamma(f)$, $f\in\mathcal{A}$, to lighten the notation. In addition, notice that the definition of the carré du champ operator is subordinate to the algebra $\mathcal{A}$. For this reason, we call the latter the \textit{carré du champ algebra}.

We now ask ourselves whether it is enough to know the Markov generator $L$ on smooth functionals on $E$ in order to describe the associated Markov semigroup $\mathbf{P}=(P_t)_{t \geq 0}$. The answer is positive, provided $L$ is a diffusion operator \citep[Definition 1.11.1]{bakry}.

\begin{definition}[Diffusion Operator]\label{diffusion}
    An operator $L$, with carré du champ operator $\Gamma$, is said to be a \textit{diffusion operator} if 
    \begin{equation}\label{diffusion-def}
        L\psi(f)=\psi^\prime(f)Lf + \psi^{\prime\prime}(f)\Gamma (f),
    \end{equation}
    for all $\psi:\mathbb R \rightarrow \mathbb R$ of class at least $\mathcal{C}^2$, and every suitably smooth functional $f$ on $E$.\footnote{Typically of class at least $\mathcal{C}^2$.}
\end{definition}

Together with the following \citep[Section 1.16.1]{bakry}, the notion of diffusion operator will allow us to link the probabilistic properties of a Markov semigroup to the geometry of the state space $E$.

\begin{definition}[Iterated Carré du Champ Operator]\label{iter-cdc}
    Let $\mathbf{P}=(P_t)_{t \geq 0}$ be a symmetric Markov semigroup with (good) state space $(E,\mathcal{F})$ and invariant reversible measure $\mu$. Let $L$ be its infinitesimal generator with $\mathbb{L}^2(\mu)$-domain $\mathcal{D}(L)$, and denote by $\mathcal{A}$ the algebra on which the carré du champ operator $\Gamma$ of $L$ is well defined. The \textit{iterated carré du champ operator} of the Markov generator $L$ is defined as
    \begin{equation}\label{iter-cdc-def}
        \mathcal{A}\ni\Gamma_2(f,g) \coloneqq  \frac{1}{2}\left[ L\Gamma(f,g) -\Gamma(f,Lg)-\Gamma(Lf,g) \right],
    \end{equation}
    for every pair $(f,g)\in \mathcal{A} \times \mathcal{A}$ such that the terms on the right hand side of \eqref{iter-cdc-def} are well defined.
\end{definition}

As for the carré du champ operator $\Gamma$, we write $\Gamma_2(f,f) \equiv \Gamma_2(f)$ to lighten the notation. Unlike $\Gamma$, $\Gamma_2$ is not always positive. As we will see more formally in \ref{riemann-manif}, $\Gamma_2$ is positive if and only if the state space $E$ has ``positive curvature''. This high-level intuition, together with the fact that many differential inequalities can be expressed via $\Gamma_2$ and $\Gamma$, justifies the following definition \citep[Definition 1.16.1]{bakry}.

\begin{definition}[Bakry-Émery Curvature]\label{b-e-curv}
    Let $\mathbf{P}=(P_t)_{t \geq 0}$ be a symmetric Markov semigroup with (good) state space $(E,\mathcal{F})$ and invariant reversible measure $\mu$. Let $L$ be its infinitesimal generator with $\mathbb{L}^2(\mu)$-domain $\mathcal{D}(L)$, and assume it is a diffusion operator. Denote by $\mathcal{A}$ the algebra on which the carré du champ operator $\Gamma$ of $L$ is well defined.\footnote{E.g. the algebra of functionals of class at least $\mathcal{C}^2$ on $E$.} Operator $L$ is said to satisfy the \textit{Bakry-Émery curvature condition} $CD(\rho,n)$, for $\rho \in \mathbb{R}$ and $n\in [1,\infty]$, if for every functional $f$ in $\mathcal{A}$, the following holds $\mu$-a.e.
    \begin{equation}\label{b-e-curv-def}
        \Gamma_2(f) \geq \rho \Gamma(f) + \frac{1}{n} (Lf)^2.
    \end{equation}
\end{definition}

Of course, equation \eqref{b-e-curv-def} becomes $\Gamma_2(f) \geq \rho \Gamma(f)$ for $n=\infty$.

\subsubsection{Riemannian Manifold}\label{riemann-manif}
When $E=\mathbb{R}^n$, or $E=M$, a complete (with respect to the Riemannian distance) Riemannian manifold (see Figure \ref{fig-riem}), the carré du champ and the iterated carré du champ operators are easily computed.

For example, if $\mathcal{A}_0$ is the class of smooth ($\mathcal{C}^\infty$) functionals with compact support\footnote{The reason for subscript $0$ in $\mathcal{A}_0$ will be clear in Section \ref{diff-mak-tr-back}.} -- i.e. such that $\text{supp}(f) = \{x \in E : f(x) \neq 0\}$ is compact in the Euclidean topology -- and $L=\Delta$ (the standard Laplacian on $\mathbb{R}^n$), then $\Gamma(f)=|\nabla f|^2$ and $\Gamma_2(f)=|\nabla\nabla f|^2$, where $\nabla\nabla f=\text{Hess}(f)=(\partial_{ij} f)_{i,j \in\{1,\ldots,n\}}$ is the Hessian of $f$. Let instead $L$ be the Laplace-Beltrami operator $\Delta_\mathfrak{g}$ on an $n$-dimensional Riemannian manifold $(M,\mathfrak{g})$. Here $\mathfrak{g}$ is the Riemannian co-metric $\mathfrak{g} \equiv \mathfrak{g}(x)=(g^{ij})_{i,j\in\{1,\ldots,n\}}=((g_{ij})_{i,j\in\{1,\ldots,n\}})^{-1}=G^{-1}$, where $G$ is the positive-definite
symmetric matrix that gives the Riemannian metric. Then, $\Gamma(f)=|\nabla f|^2$ and -- thanks to the Bochner-Lichnerowicz formula \citep[Theorem C.3.3]{bakry} -- $\Gamma_2(f)=|\nabla\nabla f|^2 + \text{Ric}(\nabla f,\nabla f)$, where $\text{Ric}=\text{Ric}_\mathfrak{g}$ is the Ricci tensor of $(M,\mathfrak{g})$.\footnote{Loosely, the Ricci tensor gives us as measure of how the geometry of a given metric tensor differs locally from that of ordinary Euclidean space. Intuitively, it quantifies how a shape is deformed as one moves along the geodesics in the Riemannian manifold.} This implies that 
$$\Gamma_2(f) \geq \rho \Gamma(f)=\rho |\nabla f|^2 \text{, } \forall f \in \mathcal{A}_0 \iff \text{Ric}(\nabla f,\nabla f) \geq \rho |\nabla f|^2,$$
that is, the Bakry-Émery curvature condition $CD(\rho,\infty)$ holds if and only if the Ricci tensor at every point is bounded from below by $\rho$. In turn, $\Gamma_2$ is positive if and only if the Ricci curvature $\text{Ric}$ of the manifold is positive. This makes more precise the idea we expressed in the previous section that $\Gamma_2$ is positive if and only if the state space $E$ has ``positive curvature''.

We now introduce the concept of elliptic diffusion operator. First, notice that when $E$ is $\mathbb{R}^n$ or a Riemannian manifold $M$, we can rewrite condition \eqref{diffusion-def} for $L$ to be a diffusion operator as

\begin{equation}\label{diffusion-def2}
    Lf=\sum_{i,j \in \{1,\ldots,n\}} g^{ij} \partial^2_{ij} f + \sum_{i=1}^n b^i \partial_i f, \quad f \in\mathcal{A}_0.
\end{equation}

Here, $\mathfrak{g}\equiv \mathfrak{g}(x)=(g^{ij}(x))_{i,j \in \{1,\ldots,n\}}$ is a smooth ($\mathcal{C}^\infty$) $n\times n$ symmetric matrix-valued function of $x \in E$ (e.g. the co-metric of $M$ at $x$), and $b\equiv b(x)=(b^i(x))_{i=1}^n$  is a smooth ($\mathcal{C}^\infty$) $\mathbb{R}^n$-valued function of $x\in E$ (e.g. a vector field). In turn, this implies that we can write the carré du champ operator as $\Gamma(f,g)=\sum_{i,j \in \{1,\ldots,n\}} g^{ij}\partial_i f \partial_j g$, $f,g \in\mathcal{A}_0$. A diffusion operator as in \eqref{diffusion-def2} is said to be \textit{elliptic} if

\begin{equation*}
    \mathfrak{g}(V,V)=\sum_{i,j \in \{1,\ldots,n\}} g^{ij}(x) V_iV_j \geq 0 \text{, } \forall V=(V_i)_{i=1}^n \in \mathbb{R}^n 
\end{equation*}
and
\begin{equation*}
    \mathfrak{g}(V,V)=0 \implies V=0.
\end{equation*}

Loosely, the fact that a diffusion operator $L$ is elliptic is important because it ensures that any solution on $[0,s] \times \mathcal{O}$ of the heat equation $\partial_t u = Lu$, $u(0,x)=u_0$, where $s \in \mathbb{R}_{>0}$ and $\mathcal{O}\subseteq \mathbb{R}^n$ is open in the Euclidean topology, is $\mathcal{C}^\infty$ on $(0,s) \times \mathcal{O}$, no matter the initial condition $u_0$. 

\subsection{Background on Diffusion Markov Triples}\label{diff-mak-tr-back}

We first introduce Diffusion Markov Triples.

\begin{definition}[Diffusion Markov Triple]\label{diff-mark-trip}
    A Diffusion Markov Triple $(E,\mu,\Gamma)$ is composed of a measure space $(E,\mathcal{F},\mu)$, a class $\mathcal{A}_0$ of bounded measurable functionals on $E$, and a symmetric bilinear operator (the carré du champ operator) $\Gamma:\mathcal{A}_0 \times \mathcal{A}_0 \rightarrow \mathcal{A}_0$, satisfying properties D1-D9 in \citet[Section 3.4.1]{bakry}. We report them here to keep the paper self-contained.
    \begin{itemize}
        \item[(DMT1)] Measure space $(E,\mathcal{F},\mu)$ is a \textit{good measure space}, that is, there is a countable family of sets which generates $\mathcal{F}$ (up to sets of $\mu$-measure $0$), and both the decomposition and bi-measure theorems apply. Measure $\mu$ is $\sigma$-finite and, when finite, assumed to be a probability.
        \item[(DMT2)] $\mathcal{A}_0$ is a vector space of bounded measurable functionals on $E$, which is dense in every $\mathbb{L}^p(\mu)$, $1\leq p <\infty$, stable under products (that is, $\mathcal{A}_0$ is an algebra) and stable under the action of smooth ($\mathcal{C}^\infty$) functions $\Psi:\mathbb{R}^k \rightarrow \mathbb{R}$ vanishing at $0$, i.e. such that $\Psi(0)=0$.
        \item[(DMT3)] The carré du champ operator $\Gamma$ is positive, that is, $\Gamma(f) \equiv \Gamma(f,f) \geq 0$, for all $f\in\mathcal{A}_0$.\footnote{The carré du champ operator $\Gamma$ also satisfies \citet[Equation (3.1.1)]{bakry}, that is, $\int_E \Gamma(g,f^2) \text{d}\mu + 2 \int_E g \Gamma(f) \text{d}\mu = 2 \int_E \Gamma(fg,f) \text{d}\mu$.} It satisfies the following diffusion hypothesis: for every smooth ($\mathcal{C}^\infty$) function $\Psi$ vanishing at $0$, and every $f_1,\ldots,f_k,g\in\mathcal{A}_0$,
        \begin{equation}\label{diff-gamma-gen}
            \Gamma(\Psi(f_1,\ldots,f_k),g)=\sum_{i=1}^k \partial_i \Psi(f_1,\ldots,f_k) \Gamma(f_i,g).
        \end{equation}
        \item[(DMT4)] For every $f \in \mathcal{A}_0$, there exists a finite constant $C(f)$ such that for every $g \in \mathcal{A}_0$,
        $$\left| \int_E \Gamma(f,g) \text{d}\mu \right| \leq C(f) \|g\|_2.$$
        The Dirichlet form $\mathcal{E}$ is defined for every $(f,g) \in \mathcal{A}_0 \times \mathcal{A}_0$ by
        $$\mathcal{E}(f,g)=\int_E \Gamma(f,g) \text{d}\mu,$$
        and we write $\mathcal{E}(f,f) \equiv \mathcal{E}(f)$ for notational convenience. The domain $\mathcal{D}(\mathcal{E})$ of the Dirichlet form $\mathcal{E}$ is the completion of $\mathcal{A}_0$ with respect to the norm $\| f\|_\mathcal{E}= [\|f\|_2^2 + \mathcal{E}(f)]^{1/2}$. The Dirichlet form $\mathcal{E}$ is extended to $\mathcal{D}(\mathcal{E})$ by continuity together with the carré du champ operator $\Gamma$.
        \item[(DMT5)] $L$ is a linear operator on $\mathcal{A}_0$ defined by and satisfying the integration by parts formula
        $$\int_E gLf\text{d}\mu = - \int_E \Gamma(f,g) \text{d}\mu,$$
        for all $f,g\in\mathcal{A}_0$. A consequence of condition (DMT3) is that
        \begin{align}\label{diff-l-gen}
        \begin{split}
            L\left( \Psi(f_1,\ldots,f_k) \right) &= \sum_{i=1}^k \partial_i \Psi(f_1,\ldots,f_k) Lf_i \\
            &+ \sum_{i,j=1}^k \partial^2_{ij} \Psi(f_1,\ldots,f_k) \Gamma(f_i,f_j),
        \end{split}
        \end{align}
        for a smooth ($\mathcal{C}^\infty$) functions $\Psi:\mathbb{R}^k \rightarrow \mathbb{R}$ such that $\Psi(0)=0$, and $f_1,\ldots,f_k\in\mathcal{A}_0$. Hence $L$ too satisfies the diffusion property as we introduced it in Definition \ref{diffusion}.
        \item[(DMT6)] For the operator L defined in (DMT5), $L(\mathcal{A}_0) \subset \mathcal{A}_0$, i.e. $Lf\in\mathcal{A}_0$, for all $f\in\mathcal{A}_0$.
        \item[(DMT7)] The domain $\mathcal{D}(L)$ of the operator $L$  is defined as the set of $f\in\mathcal{D}(\mathcal{E})$ for which there exists a finite constant $C(f)$ such that, for any $g\in \mathcal{D}(\mathcal{E})$,
        $$\left| \mathcal{E}(f,g)\right| \leq C(f) \|g\|_2.$$
        On $\mathcal{D}(L)$, $L$ is extended via the integration by parts formula for every $g\in \mathcal{D}(\mathcal{E})$. 
        \item[(DMT8)] For every $f\in\mathcal{A}_0$, $\int_E Lf \text{d}\mu=0$. 
        \item[(DMT9)] The semigroup $\mathbf{P}=(P_t)_{t \geq 0}$ is the symmetric semigroup with infinitesimal generator $L$ defined on its domain $\mathcal{D}(L)$. In general, it is sub-Markov, that is, $P_t(\mathbbm{1})\leq \mathbbm{1}$, for all $t \geq 0$, where $\mathbbm{1}$ is the constant function equal to $1$. For it to be Markov, $P_t(\mathbbm{1})=\mathbbm{1}$ must hold for all $t \geq 0$.  Moreover, $\mathcal{A}_0$ is not necessarily dense in the domain $\mathcal{D}(L)$ with respect to the domain norm $\| f\|_{\mathcal{D}(L)}= [\|f\|_2^2 + \|Lf\|_2^2]^{1/2} \equiv \|f\|_{\mathbb L^2(\mu)}+\| L f\|_{\mathbb L^2(\mu)}$. By (DMT8), we have that $\mu$ is invariant for $\mathbf{P}$. In addition, from the symmetry of $\Gamma$ we have that $\mu$ is reversible for $\mathbf{P}$.
    \end{itemize}
\end{definition}

Given Definition \ref{diff-mark-trip}, we see how $L$ is a second order differential operator, $\Gamma$ is a first order operator in each of its arguments, and $\Gamma_2$ (see Definition \ref{iter-cdc}) is a second order differential in each of its arguments.

Before recalling adjoints and essential self-adjointness, we briefly fix notation on the graph norm and closedness for unbounded operators, since we will use approximation on $\mathcal D(L)$ through a core.

Let $(H,\|\cdot\|_H)$ be a Banach space and let $A:\mathcal D(A)\subset H\to H$ be a linear operator.
The \emph{graph} of $A$ is
\[
\mathrm{Graph}(A)\coloneqq \{(f,Af): f\in\mathcal D(A)\}\subset H\times H,
\]
and the associated \emph{graph norm} on $\mathcal D(A)$ is
\[
\|f\|_{\mathcal D(A)} \coloneqq \|f\|_H + \|Af\|_H.
\]
We recall that $A$ is \emph{closed} if and only if $\mathrm{Graph}(A)$ is closed in $H\times H$.
Equivalently, $A$ is closed if and only if $(\mathcal D(A),\|\cdot\|_{\mathcal D(A)})$ is a Banach space.
 Indeed, the map $f\mapsto (f,Af)$ is an isometric embedding of $\mathcal D(A)$ into $H\times H$ endowed with
$\|(u,v)\|_{H\times H}\coloneqq \|u\|_H+\|v\|_H$, and completeness is equivalent to closedness of the range. 
In our setting, $H=\mathbb L^2(\mu)$ and the generator $ L$ of the strongly continuous Markov
semigroup $( P_t)_{t\ge 0}$ is closed; Hence, $\mathcal D( L)$ is complete for the graph norm
$\|\cdot\|_{\mathcal D( L)}$.

When we write that $\mathcal A_0$ is dense in $\mathcal D(L)$ for the graph norm, we mean that
$\mathcal A_0$ is a \emph{core} for $L$: For every $f\in\mathcal D(L)$ there exists
$(f_n)_n\subset \mathcal A_0$ such that
\[
\|f_n-f\|_{\mathbb L^2(\mu)} + \|L f_n-L f\|_{\mathbb L^2(\mu)} \rightarrow 0.
\]
This approximation is precisely what allows us to extend identities first proved on $\mathcal A_0$
(e.g. integration by parts and the Poincar\'e inequality) to all $f\in\mathcal D(L)$ by continuity.

We now introduce the concepts of adjoint operator, self-adjointness and essential self-adjointness (ESA) \citep[Section 3.4.2]{bakry}. The latter allow us in Theorem \ref{thm-main-2}, to relax the requirement in Theorem \ref{thm-main-1} that $\underline{L} $ and $\overline{L} $ are elliptic operators.

\begin{definition}[Adjoint Operator]\label{adj-oper}
    Let $(E,\mu,\Gamma)$ be a diffusion Markov triple, and consider the operator $L$ associated with it. The domain $\mathcal{D}(L^*)$ is the set of functions $f\in\mathbb{L}^2(\mu)$ such that there exists a finite constant $C(f)$ for which, for all $g\in\mathcal{A}_0$,
    $$\left| \int_E fLg \text{d}\mu \right| \leq C(f) \|g\|_2.$$
    On this domain, the \textit{adjoint operator} $L^*$ is defined  by integration by parts: For any $g\in\mathcal{A}_0$,
    $$L^*(f)=\int_E fLg \text{d}\mu = \int_E gL^*f \text{d}\mu.$$   
\end{definition}

By symmetry of $L$ -- as we have seen in \eqref{eq-symm-gen} -- it holds that (in general) $\mathcal{D}(L) \subset \mathcal{D}(L^*)$, and $L=L^*$ on $\mathcal{D}(L)$. That is, $L$ is \textit{self-adjoint} on $\mathcal{D}(L)$. Being self-adjoint on $\mathcal{D}(L)$, though, does not mean that $L=L^*$. The self-adjointness property of $L$ refers to the dual operator constructed in the same way as $L^*$, but where $\mathcal{D}(L)$ replaces $\mathcal{A}_0$. In other words, as we have seen in Definition \ref{diff-mark-trip}.(DMT9), $\mathcal{A}_0$ is not necessarily dense in $\mathcal{D}(L)$ with respect to the domain norm $\| \cdot\|_{\mathcal{D}(L)}$. This happens if and only if $\mathcal{D}(L)=\mathcal{D}(L^*)$. In this case, the extension of $L$ from
$\mathcal{A}_0$ to a larger domain as a self-adjoint operator is unique. 

\begin{definition}[Essential Self-Adjointness -- ESA]\label{esa}
    The operator $L$ is said to be essentially self-adjoint if $\mathcal{D}(L)=\mathcal{D}(L^*)$ (recall that $\mathcal{D}(L^*)$ is defined with respect to $\mathcal{A}_0$). Equivalently, $\mathcal{A}_0$ is dense in $\mathcal{D}(L)$ in the topology induced by $\| \cdot\|_{\mathcal{D}(L)}$.
\end{definition}

As we can see, the ESA property depends on the choice of $\mathcal{A}_0$. In addition, ESA holds for all elliptic operators on complete Riemannian manifolds where $\mathcal{A}_0$ is the class of smooth compactly supported functions. That is, it is always satisfied when $E$ is a Euclidean space or a complete Riemannian manifold. The other implication, though, need not hold -- that is, there are ESA operators that are not elliptic as in \ref{riemann-manif}. As we shall see, ESA alone is enough to guarantee ergodicity à la Corollary \ref{limiting-cor}.

The ESA property does not automatically hold for a diffusion Markov triple $(E,\mu,\Gamma)$. To derive sufficient conditions for it, we need to introduce the concept of an extension of $\mathcal{A}_0$ \citep[Section 3.3.1]{bakry}.

\begin{definition}[Extended Algebra $\mathcal{A}$]\label{ext-alg}
    Let $\mathcal{A}_0$ be the algebra associated with a diffusion Markov triple $(E,\mu,\Gamma)$. $\mathcal{A}$ is an algebra of measurable functions on $E$ containing $\mathcal{A}_0$, containing the constant functions, and satisfying the following requirements.
    \begin{itemize}
        \item[(i)] For all $f\in\mathcal{A}$ and all $h\in\mathcal{A}_0$, $hf\in\mathcal{A}_0$.
        \item[(ii)] For any $f\in\mathcal{A}$, if $\int_E hf \text{d}\mu \geq 0$ for all positive $h\in\mathcal{A}_0$, then $f \geq 0$.
        \item[(iii)] $\mathcal{A}$ is stable under composition with smooth ($\mathcal{C}^\infty$) functions $\Psi:\mathbb{R}^k \rightarrow \mathbb{R}$.
        \item[(iv)] The operator $L:\mathcal{A} \rightarrow \mathcal{A}$ is an extension of the infinitesimal generator $L$ on $\mathcal{A}_0$. The carré du champ operator $\Gamma$ is also defined on $\mathcal{A} \times \mathcal{A}$ as in Definition \ref{carre}, with respect to the extension of $L$.
        \item[(v)] For every $f\in\mathcal{A}$, $\Gamma(f) \geq 0$.
        \item[(vi)] The operators $\Gamma$ and $L$ satisfy \eqref{diff-gamma-gen} and \eqref{diff-l-gen}, respectively, for any smooth ($\mathcal{C}^\infty$) functions $\Psi:\mathbb{R}^k \rightarrow \mathbb{R}$.
        \item[(vii)] For every $f \in\mathcal{A}$ and every $g \in\mathcal{A}_0$, the following holds
        $$\mathcal{E}(f,g)=\int_E \Gamma(f,g) \text{d}\mu= - \int_E gLf \text{d}\mu= - \int_E fLg \text{d}\mu.$$
        \item[(viii)] For every $f\in\mathcal{A}_0$ and every $t\geq 0$, $P_tf\in\mathcal{A}$.
    \end{itemize}
\end{definition}

When $E$ is a Euclidean space or a complete Riemannian manifold, we have that $\mathcal{A}_0$ is the class of smooth ($\mathcal{C}^\infty$) compactly supported functionals on $E$, while $\mathcal{A}$ is the class of smooth ($\mathcal{C}^\infty$) functionals. We now present two concepts, connexity and weak hypo-ellipticity, that allow to ``import'' some of the geometric structure of the Riemannian manifolds considered in \ref{riemann-manif} to the more general diffusion Markov triple setting \citep[Section 3.3.3]{bakry}. In addition, as we can see in the proof of Theorem \ref{thm-main-2}, in the context of imprecise diffusion Markov tuples they are sufficient for the ESA property to hold.  

\begin{definition}[Connexity and Weak Hypo-Ellipticity]\label{conn-w-hypo-ell}
    Let $(E,\mu,\Gamma)$ be a diffusion Markov triple, and let $\mathcal{A}$ be the extension of the algebra $\mathcal{A}_0$ associated with $(E,\mu,\Gamma)$. Then, $(E,\mu,\Gamma)$ is said to be \textit{connected} if $f\in\mathcal{A}$ and $\Gamma(f)=0$ imply that $f$ is constant. This is a local property for functionals in $\mathcal{A}$. $(E,\mu,\Gamma)$ is \textit{weakly hypo-elliptic} if, for every $\lambda\in\mathbb R$, any $f\in\mathcal{D}(L^*)$ satisfying $L^*f=\lambda f$ belongs to $\mathcal{A}$.
\end{definition}

Notice also that in this more general setting, the Bakry-Émery curvature condition is the same as that in Definition \ref{b-e-curv}, where the sufficiently rich class corresponds to the extended algebra $\mathcal{A}$.

\bibliographystyle{elsarticle-harv} 
\bibliography{tropical}



\end{document}